\documentclass{article}

\RequirePackage[OT1]{fontenc}
\RequirePackage{amsthm,amsmath}

\usepackage{amsfonts, amssymb, graphicx, mathtools, epstopdf}
\usepackage{accents, subcaption, color}

\newtheorem{definition}{Definition}

\newtheorem{theorem}{Theorem}
\newtheorem{assumption}{Assumption}

\newtheorem{corollary}{Corollary}
\newtheorem{lemma}{Lemma}

\newcommand{\ip}[1]{\langle #1 \rangle}
\newcommand{\st}{{\textrm{s.t. }}}

\newcommand{\obf}{{{\bf 1}}}

\newtheorem{innercustomthm}{Theorem}
\newenvironment{customthm}[1]
  {\renewcommand\theinnercustomthm{#1}\innercustomthm}
  {\endinnercustomthm}
\newtheorem{innercustomlem}{Lemma}
\newenvironment{customlem}[1]
  {\renewcommand\theinnercustomlem{#1}\innercustomlem}
  {\endinnercustomlem}

\newcommand{\asone}{{1 }}
\newcommand{\astwo}{{2 }}
\newcommand{\thtwo}{{2 }}

\title{A Semidefinite Program for Structured Blockmodels}
\author{David Choi}

\begin{document}

\maketitle

\begin{abstract}
Semidefinite programs have recently been developed for the problem of community detection, which may be viewed as a special case of the stochastic blockmodel. Here, we develop a semidefinite program that can be tailored to other instances of the blockmodel, such as non-assortative networks and overlapping communities. We establish label recovery in sparse settings, with conditions that are analogous to recent results for community detection. In settings where the data is not generated by a blockmodel, we give an oracle inequality that bounds excess risk relative to the best blockmodel approximation. Simulations are presented for  community detection, for overlapping communities, and for latent space models.
\end{abstract}

%
%\begin{abstract}
%Semidefnite programs have recently been developed for the problem of community detection, which may be viewed a special case of the stochastic blockmodel. Here, we develop a semidefinite program that can be tailored to other instances of the blockmodel, such as non-assortative networks and overlapping communities. We establish label recovery in sparse settings, with conditions that are analogous to recent results for community detection. When the data is not generated by a blockmodel, we give an oracle inequality that bounds excess risk relative to the best blockmodel approximation. Simulations are presented for  community detection, for overlapping communities, and for latent space models.
%\end{abstract}

\section{Introduction}

The stochastic blockmodel \cite{holland1983stochastic, bickel2009nonparametric} is a popular class of models for network data in which each node is assumed to belong to a latent class. Various sub-families of the blockmodel now exist, such as community structure \cite{karrer2011stochastic}, hierarchical community structure \cite{lyzinski2015community, peixoto2014hierarchical}, and overlapping blockmodels \cite{yang2013overlapping, zhang2014detecting, zhang2014scalable}, as well as relatives such as latent space models \cite{hoff2002latent}, mixed membership \cite{airoldi2008mixed}, degree-corrected blockmodels \cite{karrer2011stochastic}, and time-varying blockmodels \cite{matias2015statistical}.

For all of these models, estimation of the latent nodal classes is an active area of research. For blockmodels, spectral methods are known to yield asymptotically consistent estimates, provided that the network is sufficiently large and dense \cite{lyzinski2014perfect, rohe2011spectral}. For the special case of community structure, it is additionally known that  specialized methods can achieve weakly consistent estimates even when spectral methods fail completely due to sparsity \cite{mossel2012stochastic, krzakala2013spectral, gao2015achieving, amini2013pseudo, zhang2015minimax}. Examples of such methods include semidefinite programming \cite{amini2014semidefinite, guedon2015community, montanari2015semidefinite} and message passing \cite{decelle2011inference, zhang2014scalable}. For other variants of the blockmodel and their relatives, estimation methods also exist but are less understood; in particular, theory analogous to that of community detection does not yet seem to exist for these cases.

To address this gap, we show in this paper that semidefinite programming can be applied not only to community detection, but also to other blockmodel sub-families as well. Specifically, we propose a semidefinite program that can be tailored to any specific instance of the blockmodel. For this program, we prove estimation bounds that are analogous to those already known for community detection, including weak consistency in the bounded degree setting. When the data is not generated from a blockmodel, the semidefinite program can be used to construct a ``de-noised'' version of the data matrix, and we provide an oracle inequality bounding its error relative to the best blockmodel approximation. 
%Simulated examples involving community structure, hierarchy, overlapping communities, and latent distance embeddings are shown.

The organization of the paper is as follows. Section \ref{sec: formulation} presents the semidefinite program. Section \ref{sec: theorem} presents a convergence analysis for sparse data. Section \ref{sec: admm} discusses numerical optimization. Section \ref{sec: examples} gives simulations results. Proofs are contained in the appendix.

\section{Problem Formulation} \label{sec: formulation}

In this section, we present the generative models that we will consider in this paper; derive a semidefinite relaxation for the combinatorial problem of estimating the latent nodal classes; and present estimators for blockmodel and non-blockmodel settings. 

\subsection{Preliminary Notation} \label{sec: notation}

Given a matrix $M \in \mathbb{R}^{nK \times nK}$, let $M^{(ij)} \in \mathbb{R}^{K \times K}$ for $i,j \in [n]$ denote its $(i,j)$th submatrix of size $K \times K$, and similarly for $x \in \mathbb{R}^{nK}$ let $x^{(i)} \in \mathbb{R}^K$ for $i=1,\ldots,n$ denote its $i$th subvector of length $K$, so that  
\[M = \left[\begin{array}{ccc} M^{(1,1)} & \cdots & M^{(1,n)} \\ \vdots & \ddots & \vdots \\ M^{(n,1)} & \cdots & M^{(n,n)}\end{array}\right] \qquad \textrm{ and } \qquad x = \left[\begin{array}{c} x^{(1)} \\ \vdots \\ x^{(n)}\end{array}\right].\]
We will use $M^{(ij)}_{ab}$ to denote the $(a,b)$th entry of the submatrix $M^{(ij)}$, and likewise use $x^{(i)}_a$ to denote the $a$th entry of the subvector $x^{(i)}$. This implies that $M^{(ij)}_{ab} = M_{a + (i-1)K, b + (j-1)K}$ and $x^{(i)}_a = x_{a + (i-1)K}$.

\subsection{Generative Models} \label{sec: models}

\paragraph{Stochastic Blockmodel} Let $A \in \{0,1\}^{n\times n}$ denote the symmetric adjacency matrix of a undirected network with $n$ nodes. In a stochastic blockmodel with $K$ classes, each node has a random latent class $z_i \in [K]$, and 
the upper triangular entries of $A$ are independent Bernoulli when conditioned on $z$:
\begin{align} \label{eq: sbm}
z_i & \sim \textrm{Discrete}(\pi) \\
A_{ij} & \sim \textrm{Bernoulli}(\theta_{z_i,z_j}), \qquad i,j \in [n],\, i<j \\
\nonumber A_{ji} & = A_{ij},
\end{align}
where $\pi$ is a probability distribution over $[K]$ giving the expected class frequencies, and $\theta \in [0,1]^{K \times K}$ is a symmetric matrix that gives the connection probabilities between each class type.

\paragraph{General Model} Under a more general model, $A \in \{0,1\}^{n \times n}$ is a random matrix generated by 
\begin{align} \label{eq: bernoulli}
A_{ij} & \sim \operatorname{Bernoulli}(P_{ij}) & i < j \\
\nonumber A_{ji} & = A_{ij},
\end{align}
where $P \in [0,1]^{n \times n}$ is symmetric and satisfies $P_{ii}=0$ for $i\in [n]$. It can be seen that the stochastic blockmodel is a special case of \eqref{eq: bernoulli}, where $P_{ij} = \theta_{z_i,z_j}$. 

\subsection{Semidefinite program} \label{sec: sdp}

We will assume that $A \in \{0,1\}^{n \times n}$ is observed, and the estimation task is to find $z \in [K]^n$ maximizing the generic combinatorial problem
\begin{equation}\label{eq: combinatoric}
\max_{z \in [K]^n} \sum_{i=1}^n \sum_{j=1}^n f_{ij}(z_i,z_j),
\end{equation}
under some choice of objective functions $f_{ij}:[K]^2\mapsto \mathbb{R}$ for $i,j \in [n]$. In this paper, we will let $\{f_{ij}\}$ equal the likelihood function
\begin{equation}\label{eq: likelihood}
f_{ij}(z_i,z_j) = A_{ij} \log \hat{\theta}_{z_iz_j} + (1-A_{ij}) \log (1-\hat{\theta}_{z_i,z_j}),
\end{equation}
in which case \eqref{eq: combinatoric} finds the maximum likelihood assignment of $z$ under a specified parameter matrix $\hat{\theta} \in [0,1]^{K\times K}$. Note that $\hat{\theta}$ may differ from the actual generative model for $A$. 

Optimizing \eqref{eq: combinatoric} is not computationally tractable, so we will relax it into a semidefinite program. Let $F \in \mathbb{R}^{nK \times nK}$ denote a matrix with submatrices $F^{(ij)} \in \mathbb{R}^{K \times K}$ given by
\begin{equation}\label{eq: F}
F^{(ij)}_{ab} = f_{ij}(a,b) \qquad \qquad a,b \in [K],
\end{equation}
so that \eqref{eq: combinatoric} can be expressed as
\begin{align} \label{eq: generic}
\max_{z \in [K]^n} \sum_{i=1}^n \sum_{j=1}^n e_{z_i}^T F^{(ij)} e_{z_j},
\end{align}
where $e_1,\ldots,e_K \in \{0,1\}^K$ denote the canonical basis in $\mathbb{R}^K$. This can be further rewritten as
\begin{align*}
	\max_{z \in [K]^n, x \in \{0,1\}^{nK}} & \ip{F, xx^T} \quad \textrm{subject to } x = \left[\begin{array}{c} e_{z_1} \\ \vdots \\ e_{z_n} \end{array}\right].
\end{align*}
This suggests the following semidefinite program relaxation, where $xx^T$ is approximated by a positive semidefinite matrix $X$:
\begin{align} \label{eq: sdp}
\max_{X \in \mathbb{R}^{nK \times nK}}& \ip{F, X} \\
\nonumber \st & X \succeq 0, X \geq 0 \\
\nonumber & \sum_{a=1}^K \sum_{b=1}^K X^{(ij)}_{ab} = 1 \qquad \forall\ i,j \in [n],
\end{align}
where $X \succeq 0$ denotes that $X$ positive semidefinite, and $X \geq 0$ denotes that $X$ is elementwise non-negative. For any feasible $X \in \mathbb{R}^{nK \times nK}$, each submatrix $X^{(ij)} \in \mathbb{R}^{K \times K}$ is nonnegative and sums to one, and can be viewed as a relaxed version of the indicator matrix $e_{z_i}e_{z_j}^T$ encoding the class pair $(z_i,z_j)$.

\subsection{Matrix Denoising and Blockmodel Estimation} \label{sec: estimation}

Let $\hat{X} \in [0,1]^{nK \times nK}$ denote the solution to the semidefinite program \eqref{eq: sdp}, and let $A$ be generated from the general model \eqref{eq: bernoulli}, with generative matrix $P \in [0,1]^{n \times n}$. Let $\hat{P} \in [0,1]^{n \times n}$ denote a ``MAP estimate'' of $P$, constructed by treating each submatrix $\hat{X}^{(ij)}$ as a probability distribution over $[K]^2$:
\begin{align}\label{eq: Phat MAP}
    \hat{P}_{ij} &= \arg \max_{\vartheta \in [0,1]} \sum_{a=1}^K \sum_{b=1}^K \hat{X}_{ab}^{(ij)}\cdot 1\{\hat{\theta}_{ab} = \vartheta \}
\end{align}
Alternatively, let $\tilde{P}$ denote a randomized estimate of $P$, where each dyad is an independent random variable with distribution 
\begin{align} \label{eq: Phat randomized}
	\tilde{P}_{ij} & = \hat{\theta}_{ab} \textrm{ with probability } \hat{X}^{(ij)}_{ab} \qquad  \quad \forall\, i < j \\
\nonumber	\tilde{P}_{ji} &= \tilde{P}_{ij}.
\end{align}
Let $\hat{z} \in [K]^n$ denote the cluster labels found by spectral clustering of $\hat{P}$ -- i.e., applying $K$-means to the first $K$ eigencoordinates of $\hat{P}$. If $A$ is generated by a blockmodel, then the generative $P$ will be block structured, with blocks induced by $\theta$ and $z$. In this case, we will use $\hat{z}$ to estimate $z$, up to label-switching permutation.

To estimate $\theta$ up to permutation, let $\hat{\theta}^{\textrm{est}} \in [0,1]^{K \times K}$ denote the matrix of between-block densities induced by $\hat{z}$,
\begin{equation}\label{eq: hat theta est}
\hat{\theta}^{\textrm{est}}_{ab} = \frac{\sum_{i=1}^n \sum_{j=1}^n A_{ij} 1\{\hat{z}_i=a,\hat{z}_j=b\}}{\sum_{i=1}^n \sum_{j=1}^n 1\{\hat{z}_i=a,\hat{z}_j=b\}} \qquad a,b \in [K].
\end{equation}

\subsection{Discussion} \label{sec: related work}

\paragraph{Related Work} Semidefinite programs have been used for community detection in \cite{amini2014semidefinite, guedon2015community, montanari2015semidefinite}, as well as \cite{cai2015robust} which allowed for outliers, and \cite{chen2015convexified} which allowed for degree heterogeneity. In each of these works, the network is required to exhibit assortatitve block structure. For the general model \eqref{eq: bernoulli} without such restrictions, estimation has been considered in \cite{chatterjee2015matrix}, but only for dense settings. To the best of our knowledge, networks that are both non-assortative and sparse (i.e., bounded degree), such as those presented in Sections \ref{sec: overlapping} and \ref{sec: hoff}, have not been considered in previous work. 

Additionally, the semidefinite program presented here bears resemblance to one appearing in \cite{zhao1998semidefinite} for lower bounding the optimal objective value of the quadratic assignment problem (without finding a feasible solution)\footnote{which can be used as an intermediate step in a branch-and-bound algorithm.}, and also recent work on estimating pairwise alignments beween objects \cite{bandeira2015non}.

% Community structure has been defined both as a special case of the blockmodel \cite{guedon2015community}, but also more broadly as assortative $P$ under the general model \eqref{eq: bernoulli}, i.e., $P_{ij} > P_{i'j'}$ for all $i,j,i',j'$ such that $z_i = z_j$ and $z_{i'} \neq z_{j'}$ \cite{gao2015achieving}.
%
%
% community detection SDP, USVT for matrix denoising, semidefinite program for branch and bound of QAP, traveling salesman semidefinite program special case using association schemes. work by Alfonso Bandeiras. Future work: degree correction (Xiaodong Li), and estimation of $\theta$ - perhaps alternating between label estimation and parameter estimation?

\section{Convergence Analysis} \label{sec: theorem}

In this section, we analyze the solution of the semidefinite program \eqref{eq: sdp} for both matrix denoising and label recovery. Analogous to existing results for community detection, our results will imply weak consistency (i.e., performance better than random guessing) in the regime where the average degree of $A$ is asymptotically bounded above some constant, and consistency of $\hat{z}$, as well as vanishing excess risk of $\hat{P}$, when the average degree $\rightarrow \infty$. 

The organization of this section is the following: Section \ref{sec: thm preliminaries} defines basic notation; Section \ref{sec: assumptions} states the required conditions; and Section \ref{sec: theorems} presents the convergence results, which are proven in Appendix \ref{sec: appendix A}.

\subsection{Preliminaries} \label{sec: thm preliminaries}

The following notation will be used. Given $A \in \{0,1\}^{n \times n}$ that is generated by \eqref{eq: bernoulli} under some $P \in [0,1]^{n \times n}$, let $\rho$ denote its expected density 
\[ \rho = \frac{1}{n(n-1)} \sum_{ij} P_{ij}. \]
Given $\hat{\theta} \in [0,1]^{K \times K}$, let $F \in \mathbb{R}^{nK \times nK}$ denote the objective function of the semidefinite program \eqref{eq: sdp}, with submatrices $F^{(ij)} \in \mathbb{R}^{K \times K}$ given by
\begin{equation}\label{eq: F theorem}
F^{(ij)}_{ab} = A_{ij} \log \hat{\theta}_{ab} + (1- A_{ij}) \log (1 - \hat{\theta}_{ab}) \qquad i,j \in [n],\, a,b \in [K].
\end{equation}
Let $\bar{F} \in \mathbb{R}^{nK \times nK}$ denote a idealized version of $F$ in which $A$ is replaced by its expectation $P$, with submatrices $\bar{F}^{(ij)} \in \mathbb{R}^{K \times K}$ given by
\begin{equation} \label{eq: Fbar theorem}
\bar{F}^{(ij)}_{ab} = P_{ij} \log \hat{\theta}_{ab} + (1- P_{ij}) \log (1-\hat{\theta}_{ab}) \qquad i,j \in [n],\, a,b \in [K].
\end{equation}
Let $\mathcal{X} \subset \mathbb{R}^{nK \times nK}$ denote the feasible region of the semidefinite program \eqref{eq: sdp}:
\begin{equation}\label{eq: Xscr}
\mathcal{X} = \left\{X \in \mathbb{R}^{nK \times nK}: X \succeq 0, X \geq 0, \textrm{ and } \sum_{a=1}^K 
\sum_{b=1}^K X^{(ij)}_{ab} = 1 \textrm{ for all } i,j \in [n]\right\}.
\end{equation}
Let $\hat{X}$ denote any solution to the semidefinite program \eqref{eq: sdp}, which can be written as 
\[ \textrm{maximize } \ip{F,X} \textrm{ over all } X \in \mathcal{X}.\]
For any matrix $B \in [0,1]^{K \times K}$, let the function $\mathcal{Q}_{B}$ be given by
%map $[K]^2$ to the power set of $[K]^2$, and identify the subsets of $[K]^2$ that have equal values in $B$:
\[ \mathcal{Q}_{B}(a,b) = \{(c,d) \in [K]^2: B_{cd} = B_{ab}\} \qquad \qquad a,b \in [K],\]
so that $\mathcal{Q}_B$ identifies the subsets of $[K]^2$ that have equal values under $B$.

\subsection{Assumptions}\label{sec: assumptions}

Assumption \ref{as: rho} and \ref{as: thetahat} will apply when $\tilde{P}$ is used to estimate the general model \eqref{eq: bernoulli}. Assumption \ref{as: rho} requires $P$ to have density $\rho$ exceeding $1/n$. Assumption \ref{as: thetahat} bounds the entries of $\hat{\theta}$ to differ from $\rho$ by roughly at most a constant factor.

\begin{assumption} \label{as: rho}
Let $A \in \{0,1\}^{n \times n}$ be generated by \eqref{eq: bernoulli}, with $P \in [0,1]^{n \times n}$ (which evolves with $n$) satisfying $\rho \geq 1/n$.
\end{assumption}

\begin{assumption} \label{as: thetahat}
The matrix $\hat{\theta} \in [0,1]^{K \times K}$ (which may evolve with $n$) satisfies $\frac{\rho}{c} \leq \frac{\hat{\theta}_{ab}}{1-\hat{\theta}_{ab}} \leq c\rho$ for some $c > 0$ and all $n$ and $a,b \in [K]$.
\end{assumption}

Assumptions \ref{as: sbm1} and \ref{as: sbm2} will apply when $A$ is generated by a stochastic blockmodel, and are sufficient to show that $\hat{z}$ converges to the true $z$, up to label switching. Assumption \ref{as: sbm1} describes a parametrization that is commonly used for sparse blockmodels. Assumption \ref{as: sbm2} places bounds on the misspecification between $\hat{\theta}$ and $\theta^*$ in the sparse blockmodel setting. 

\begin{assumption}\label{as: sbm1}
For all $n$, let $A \in \{0,1\}^{n \times n}$ and $z \in [K]^n$ be generated by a stochastic blockmodel with parameters $(\pi, \theta^*)$. Let $\pi$ be constant, and let $\theta^* = \alpha B^*$, where $\alpha \in \mathbb{R}$ satisfies $\alpha \rightarrow 0$ and $\alpha > 1/n$, and $B^* \in \mathbb{R}_+^{K \times K}$ is constant, rank $K$, and satisfies
\begin{equation} \label{eq: canonical sbm}
\sum_{a=1}^K \sum_{b=1}^K \pi_a \pi_b B^*_{ab} = 1.
\end{equation}
\end{assumption}

\begin{assumption}\label{as: sbm2}
Let $\hat{\theta} = \hat{\alpha} \hat{B}$, where $\hat{\alpha} = \frac{1}{n(n-1)}\sum_{ij} A_{ij}$, and $\hat{B} \in \mathbb{R}_+^{K\times K}$ is a fixed matrix such that $\hat{B}$ and $B^*$ satisfy
\begin{align} 
	\mathcal{Q}_{\hat{B}}(a,b) = \mathcal{Q}_{B^*}(a,b) \qquad & \qquad \forall a,b \in [K] \label{eq: sbm1} \\
	B_{ab}^* \log \frac{\hat{B}_{ab}}{\hat{B}_{cd}} - (\hat{B}_{ab} - \hat{B}_{cd}) > 0 \qquad & \qquad \forall a,b \in [K] \textrm{ and } c,d \notin \mathcal{Q}_{B^*}(a,b). \label{eq: sbm2}
\end{align}
\end{assumption}
Assumption \ref{as: sbm2} states that $\hat{B}$ and $B^*$ need not have identical values, but should have the same structure (as given by $\mathcal{Q}_{\hat{B}}$ and $\mathcal{Q}_{B^*}$). Additionally, for all $a,b \in [K]$, the entry $\hat{B}_{ab}$ should be the closest element of $\hat{B}$ to $B^*_{ab}$, in terms of the Bregman divergence associated with the Poisson likelihood.

\subsection{Results} \label{sec: theorems}

Theorem \ref{th: Ptilde} holds when $A$ is generated from the general model \eqref{eq: bernoulli}, including non-blockmodels. It gives an oracle inequality on the quality of the randomized estimate $\tilde{P}$ given by \eqref{eq: Phat randomized}, relative to the best blockmodel approximation to the generative $P$.

\begin{theorem}\label{th: Ptilde}
	Let Assumptions \ref{as: rho} and \ref{as: thetahat} hold. Let $\tilde{P}$ denote the randomized estimate of $P$ given by \eqref{eq: Phat randomized}. Then
	\[\frac{1}{n^2\rho} \sum_{i=1}^n \sum_{j=1}^n KL(P_{ij}, \tilde{P}_{ij}) \leq \frac{1}{n^2\rho} \min_{\hat{z} \in [K]^n} \sum_{i=1}^n \sum_{j=1}^n KL(P_{ij}, \hat{\theta}_{\hat{z}_i,\hat{z}_j}) + O_P\left(\frac{1}{\sqrt{n \rho}}\right),\]
with $O_P(\cdot)$ having constant terms depending only on $K$ and $c$ (which appears in Assumption \ref{as: thetahat}).
\end{theorem}

Theorem \ref{th: sbm} assumes the sparse blockmodel setting of Assumptions \ref{as: sbm1} and \ref{as: sbm2}, and shows that both $\hat{P}$ and the randomized estimate $\tilde{P}$ asymptotically recover $P$, with vanishing fraction of incorrect values.

\begin{theorem} \label{th: sbm}
Let Assumptions \ref{as: sbm1} and \ref{as: sbm2} hold. Let $\hat{P}$ denote the estimate of $P$ given by \eqref{eq: Phat MAP}. Then 
\begin{equation}\label{eq: th sbm 1}
\frac{1}{n^2} \sum_{i=1}^n \sum_{j=1}^n 1\{\hat{P}_{ij} \neq \hat{\theta}_{z_i z_j}\} = O_P\left(\frac{1}{\sqrt{n \alpha}}\right),
\end{equation}
with the same result if $\hat{P}$ is replaced by $\tilde{P}$ given by \eqref{eq: Phat randomized}.
\end{theorem}

% Corollary \ref{cor: davis kahan} states that the eigencoordinates of $\hat{P}$ or $\tilde{P}$ will approximate those of $P$ (up to a unitary transform) when Theorem \ref{th: sbm} holds. This suggests that $\hat{z}$, which is computed by spectral clustering of $\hat{P}$, will converge to $z$ up to label permutation. Corollary \ref{cor: davis kahan} follows from a version of the Davis Kahan Theorem \cite[Th. 4]{yu2015useful}.
%
% \begin{corollary} \label{cor: davis kahan}
% Let Assumptions \ref{as: sbm1} and \ref{as: sbm2} hold, and let $V=(v_1,\ldots,v_K)$ denote the eigenvectors of $P$. Let $\hat{V} = (\hat{v}_1,\ldots,\hat{v}_K)$ denote the eigenvectors of the $K$ largest eigenvalues (in absolute value) of $\hat{P}$ or $\tilde{P}$. Let $\mathcal{O}$ denote the set of $K \times K$ orthogonal matrices. It holds that
% \[  \min_{O \in \mathcal{O}} \|\hat{V}O - V\|_F^2 \leq O_P\left(\frac{1}{\sqrt{n \alpha}}\right).\]
% \end{corollary}

Corollary \ref{cor: eigvec} follows from Theorem \ref{th: sbm}, and states that the eigencoordinates from $\hat{P}$, which are used to compute $\hat{z}$, converge (up to a unitary transformation) to those of $P$. As $P$ is block structured, this suggests that $\hat{z}$ will converge to $z$ up to a permutation of the labels. It is proven in Appendix \ref{sec: appendix C}, and is a direct application of the Davis-Kahan theorem in \cite[Th. 4]{yu2015useful}.

\begin{corollary} \label{cor: eigvec}
Let Assumptions \ref{as: sbm1} and \ref{as: sbm2} hold, and let $V=(v_1,\ldots,v_K)$ denote the eigenvectors of $P$. Let $\hat{V} = (\hat{v}_1,\ldots,\hat{v}_K)$ denote the eigenvectors of the $K$ largest eigenvalues (in absolute value) of $\hat{P}$ or $\tilde{P}$. Let $\mathcal{O}$ denote the set of $K \times K$ orthogonal matrices. It holds that
\[  \min_{O \in \mathcal{O}} \|\hat{V}O - V\|_F^2 \leq O_P\left(\frac{1}{\sqrt{n \alpha}}\right).\]
\end{corollary}

\paragraph{Remark} Weak consistency in the bounded degree regime is implied by each of the results, in that the estimation error is bounded away from the performance of random guessing, provided that $\lim_{n\rightarrow \infty} \rho \geq c_0/n$ (for Theorem \ref{th: Ptilde}) or $\lim_{n\rightarrow \infty} \alpha \geq c_0/n$ (for Theorem \ref{th: sbm}), for some constant $c_0$. 

\section{Numerical optimization} \label{sec: admm}

The semidefinite program \eqref{eq: sdp} can be solved by the alternating direction method of multipliers (ADMM) \cite{boyd2011distributed}. While this semidefinite program is much larger than those previously introduced for community detection, speedups can often be achieved by exploiting problem structure, resulting in competitive computation times. 

To solve \eqref{eq: sdp} using ADMM, we introduce the decision variable $X \in \mathbb{R}^{nK \times nK}$ and auxiliary variables $Y,W,U,V \in \mathbb{R}^{nK \times nK}$, which are all initialized to zero and evolve according to the following rule:% (see the appendix for a derivation):
\begin{align}
   			X^{t+1} &= \Pi_\mathcal{X}\left(\frac{1}{2}\left(W^t - U^t + Y^t - V^t + \frac{1}{\rho}F\right)\right) \label{eq: admm}\\
\nonumber	W^{t+1} &= \max(0, X^{t+1} + U^t) \\
\nonumber	Y^{t+1} &= \Pi_{\mathbb{S}_+}(X^{t+1} + V^t) \\
\nonumber	U^{t+1} &= U^t + X^{t+1} - W^{t+1} \\
\nonumber	V^{t+1} &= V^t + X^{t+1} - Y^{t+1},
\end{align}
where $F \in \mathbb{R}^{nK \times nK}$ is given by \eqref{eq: F}, $\rho \in \mathbb{R}$ is a positive stepsize parameter, the operator $\Pi_{\mathbb{S}_+}$ denotes projection onto the set of positive semidefinite matrices, and the operator $\Pi_\mathcal{X}$ denotes projection onto the affine subspace of all matrices $X \in \mathbb{R}^{nK \times nK}$ satisfying the linear constraints
\begin{align*}
		  \sum_{a=1}^K \sum_{b=1}^K X^{(ij)}_{ab} = 1 & \qquad \forall i,j \in [n]. 
\end{align*}

The slowest step of the ADMM iterations is the computation of $Y^{t+1} = \Pi_{\mathbb{S}_+}(X^{t+1} + V^t)$, which requires the eigendecomposition of an $nK \times nK$ matrix. In comparison, semidefinite programs for community detection require only the decomposition of an $n \times n$ matrix in each ADMM iteration. However, in many settings of interest, $X^{t+1} + V^t$ will be highly structured and have fast eigendecomposition methods. 

In particular, let the $K \times K$ submatrices $[X^{t+1} + V^t]^{(ij)}$ for $i,j \in [n]$ be symmetric and share a common set of orthonormal eigenvectors $v_1,\ldots,v_K$. It then holds for some $\ell \leq K$ and some partition $S_0,\ldots,S_\ell$ of $[K]$ that
\begin{equation} \label{eq: structured Y}
	[X^{t+1} + V^t]^{(ij)} = \sum_{l=0}^\ell \lambda_{ijl} \sum_{j \in S_l} v_j v_j^T \qquad \qquad \forall\, i,j \in [n],
\end{equation}
where $\lambda_{ijl}$ is the eigenvalue corresponding to the eigenvectors $\{v_j: j \in S_l\}$. In this case, $Y^{t+1}$ can be computed in the following manner:
\begin{enumerate}
	\item Find the eigendecomposition for each submatrix $[X^{t+1} + V^t]^{(ij)} \in \mathbb{R}^{K \times K}$, yielding eigenvectors $v_1,\ldots,v_K$, sets $S_0,\ldots,S_\ell$ partitioning $[K]$, and eigenvalues $\{\lambda_{ijl}\}$ satisfying \eqref{eq: structured Y}. 
	\item Let $\Lambda_0,\ldots,\Lambda_\ell \in \mathbb{R}^{n \times n}$ and $E_0,\ldots,E_\ell \in \mathbb{R}^{K \times K}$ be given by
\begin{equation*} 
\Lambda_l(i,j) = \lambda_{ijl} \qquad \textrm{ and } \qquad E_l = \sum_{j \in S_l} v_j v_j^T \qquad \qquad l=0,\ldots,\ell,
\end{equation*}
so that
\begin{equation*}
	X^{t+1} + V^t = \sum_{l=0}^\ell \Lambda_l \otimes E_l.
\end{equation*}
    \item Return
\begin{equation}\label{eq: Pi_S}
\Pi_{\mathcal{S}_+}(X^{t+1} + V^t) = \sum_{l=0}^\ell \Pi_{\mathcal{S}_+}(\Lambda_l) \otimes E_l,
\end{equation}
which holds because the matrices $E_0,\ldots,E_\ell$ are orthogonal and positive semidefinite.
\end{enumerate}	
Since $\Lambda_0,\ldots,\Lambda_\ell \in \mathbb{R}^{n \times n}$, the above steps require $\ell+1$ eigendecompositions of $n \times n$ matrices. As the computation time of eigendecomposition scales cubically in the matrix size, the resulting speedup can be quite significant in practice; for example, if $K=10$ and $\ell = \log_2 K$, the speedup is a factor of roughly $200$ (e.g., going from 3 hours to $\leq 1$ minute to solve the semidefinite program).

When will the decomposition \eqref{eq: structured Y} hold? A sufficient condition is that $\hat{\theta}$ is in the span of an {\it association scheme}, which is evidently a fundamental concept in combinatorics \cite{goemans1999semidefinite, graham1995handbook}, and is defined as follows:
\begin{definition}\cite[Sec 2]{goemans1999semidefinite}\label{def: association scheme}
A set of binary-valued and symmetric matrices $B_0,\ldots,B_\ell$ form an association scheme if the following properties hold:
\begin{enumerate}
	\item $B_0 = I$
	\item $\sum_{i=0}^\ell B_i = \obf \obf^T$
	\item For all $i,j$, it holds that $B_iB_j \in \operatorname{span}(B_0,\ldots,B_\ell)$
\end{enumerate}
\end{definition}

The key property of association schemes is the following result, which is Theorem 2.6 in \cite{bailey2004association}, or follows from E1-E4 in \cite{goemans1999semidefinite}. The result states that $B_0,\ldots,B_\ell$ share a common set of eigenvectors.

\begin{lemma}[\cite{goemans1999semidefinite, bailey2004association}] \label{le: association}
Let $B_0,\ldots,B_\ell \in \mathbb{R}^{K \times K}$ form an association scheme. Then $B_0,\ldots,B_\ell$ have the same eigenvectors $v_1,\ldots,v_K$, and for some partition $S_0,\ldots,S_\ell$ of $[K]$ and scalars $\{\lambda_j^{(i)}\}_{i,j=0}^\ell$ it holds for that
		\begin{equation} \label{eq: association}
		B_i = \sum_{j=0}^\ell \lambda_j^{(i)} \sum_{k \in S_j} v_k v_k^T, \qquad \qquad  i=0,\ldots,\ell.
		\end{equation}
Additionally, the vector $1$ will be one of the eigenvectors.
\end{lemma}

Theorem \ref{th: association2}, which is proven in Appendix \ref{sec: appendix B}, states that if the parameter matrix $\hat{\theta} \in \mathbb{R}^{K \times K}$ is in the span of an association scheme, then the decomposition \eqref{eq: structured Y} will hold, with $\ell$ equal to the number of matrices in the association scheme, and with precomputable $\{S_l\}$ and $v_1,\ldots,v_K$:

\begin{theorem}\label{th: association2}
Let $F$ be defined as \eqref{eq: F theorem}, for any symmetric $A \in \{0,1\}^{n \times n}$ and $\hat{\theta} \in \operatorname{span}(B_0,\ldots,B_\ell)$, where $\{B_i\}_{i=0}^\ell$ form an association scheme. Let $X,W,Y,U,V \in \mathbb{R}^{nK \times nK}$ be initialized to zero and evolve by the ADMM equations \eqref{eq: admm}. Then \eqref{eq: structured Y} holds for all $t$, with $v_1,\ldots,v_K$ and $S_0,\ldots,S_\ell$ satisfying \eqref{eq: association}. 

%Then for all $i,j\in [n]$, the submatrices $X^{(ij)}, W^{(ij)}, Y^{(ij)}, U^{(ij)}, V^{(ij)} \in \operatorname{span}(B_0,\ldots,B_\ell)$ for all iterations $t$. Also \eqref{eq: structured Y} holds for all $t$, with $v_1,\ldots,v_K$ and $S_0,\ldots,S_\ell$ satisfying \eqref{eq: association}. 

\end{theorem}

In Section \ref{sec: examples}, we will give three examples of semidefinite programs in which $\hat{\theta}$ can be shown to belong to  association schemes, with $\ell = 1, \ell=2$, and $\ell \leq K$ respectively. We remark that association schemes were originally invented by statisticians for the study of experiment design \cite{bailey2003designs, bailey2004association}, and that they have been used with semidefinite programming in \cite{de2008semidefinite} to lower bound the optimal objective value of the traveling salesperson problem (without finding a feasible solution).

% AOS,AOAS: If there are supplements please fill:
%\begin{supplement}[id=suppA]
%  \sname{Supplement A}
%  \stitle{Title}
%  \slink[doi]{10.1214/00-AOASXXXXSUPP}
%  \sdatatype{.pdf}" 
%  \sdescription{Some text}
%\end{supplement}

\section{Simulation results} \label{sec: examples}

To illustrate the usage and performance of the semidefinite program, we show simulated results for three examples: community structure, overlapping communities, and latent space models. 
%In the overlapping community setting, $A$ is generated by a blockmodel where it is of interest to not only estimate $z$ and $\theta$ up to a label-switching permutation, but to also estimate the permutation of the labels as well, to identify which classes of $z$ represent overlapping memberships that are ``close'' to each other. A post-hoc heuristic is presented to accomplish this, which succeeds contingent on correct estimation of $z$ up to label switching.  LAtent space models...

\subsection{Community Structure}

In the best-understood blockmodel setting, $A$ is generated by $\theta \in [0,1]^{K \times K}$ which has parameters $\gamma_0, \gamma_1 \in [0,1]$, and equals
\begin{align}\label{eq: community model}
	\theta = \gamma_0 I + \gamma_1 (\obf \obf^T - I).
\end{align}
In this model, nodes connect with probability $\gamma_0$ if they are in the same class, and probability $\gamma_1$ if they are in different classes.

Under \eqref{eq: community model}, the parameter matrix $\theta$ can be written as 
\[\theta = \gamma_0 B_0 + \gamma_1 B_1, \]
 for $B_0 = I$ and $B_1 = \obf \obf^T - I$. By manual verification, this can be seen to satisfy the requirements given in Definition \ref{def: association scheme} for an association scheme with $\ell=1$ so that fast methods can be used to evaluate the ADMM iterations. 

Figure \ref{fig: community} shows an adjacency matrix $A$ generated by \eqref{eq: community model}, the estimated class labels $\hat{z}$ found by spectral clustering of $\hat{P}$ (as discussed in Section \ref{sec: estimation}), and the estimates found by direct spectral clustering of $A$. In this instance, the semidefinite program yields nearly perfect class estimates, while spectral clustering of $A$ fails due to its sparsity. 

Figure \ref{fig: batch community} shows the average simulated performance, over a range of values for the network size $n$ and average degree $n\rho$. We see that for $n \rho \geq 10$, the semidefinite program has roughly constant error if the average degree is fixed, which is consistent with Theorem \ref{th: sbm} (and other known results for community detection). In contrast, the misclassification rate for spectral clustering of $A$ increases with the sparsity of the graph. This is exemplified by subplot (d), where spectral clustering of $A$ performs well for small networks but degrades severely as $n$ increases.

% Theorem \ref{th: association2} thus implies that \eqref{eq: structured Y} holds, so that each submatrix of $X^{t+1} + V^t$ can also be written as 
%\[ [X^{t+1} + V^t]^{(ij)} = a_{ij} I + b_{ij} (\obf \obf^T - I),\]
%for some set of scalars $\{a_{ij}\},\{b_{ij}\}$. This in turn implies that each submatrix $[X^{t+1} + V^t]^{(ij)}$ has two distinct eigenvalues,
%\[ \lambda_{ij0} = a_{ij} + b_{ij}(K-1)  \qquad \textrm{ and } \lambda_{ij1} = a_{ij} - b_{ij},\]
%corresponding to $E_0 = \obf \obf^T/K$ and $E_1 = I - \obf\obf^T/K$ respectively. 

\begin{figure} 
    \centering
    \begin{subfigure}[b]{0.24\textwidth}
        \includegraphics[width=\textwidth]{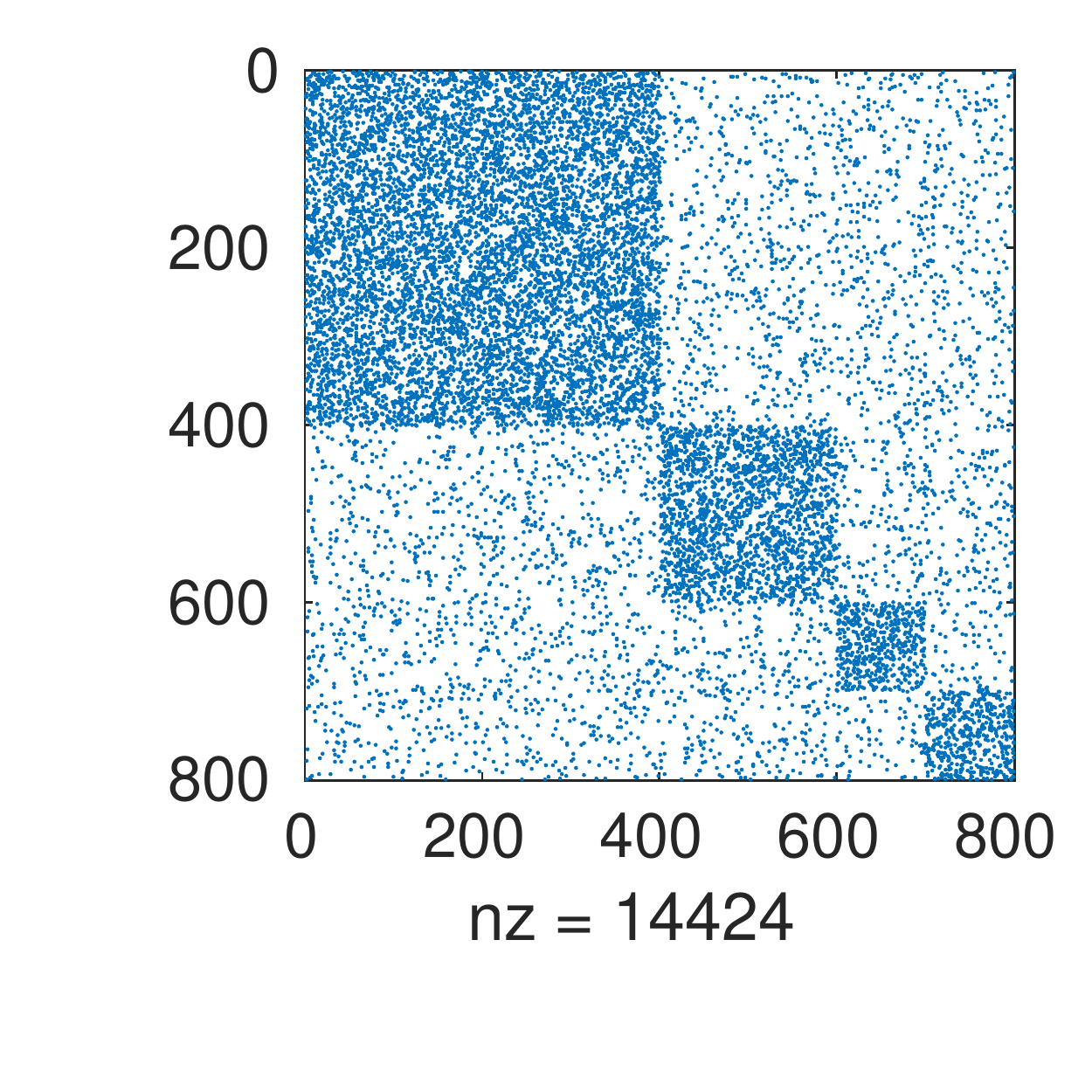}
        \caption{Adj. Matrix $A$} % {\textcolor{white}{a a a a}}\phantom{hi there} }
    \end{subfigure}
%    \begin{subfigure}[b]{0.24\textwidth}
%        \includegraphics[width=\textwidth]{figures/sbm1a.pdf}
%        \caption{Estimated $\hat{Z}$ {\textcolor{white}{a a a a}} \phantom{whoa bubba}}
%    \end{subfigure}
    \begin{subfigure}[b]{0.24\textwidth}
        \includegraphics[width=\textwidth]{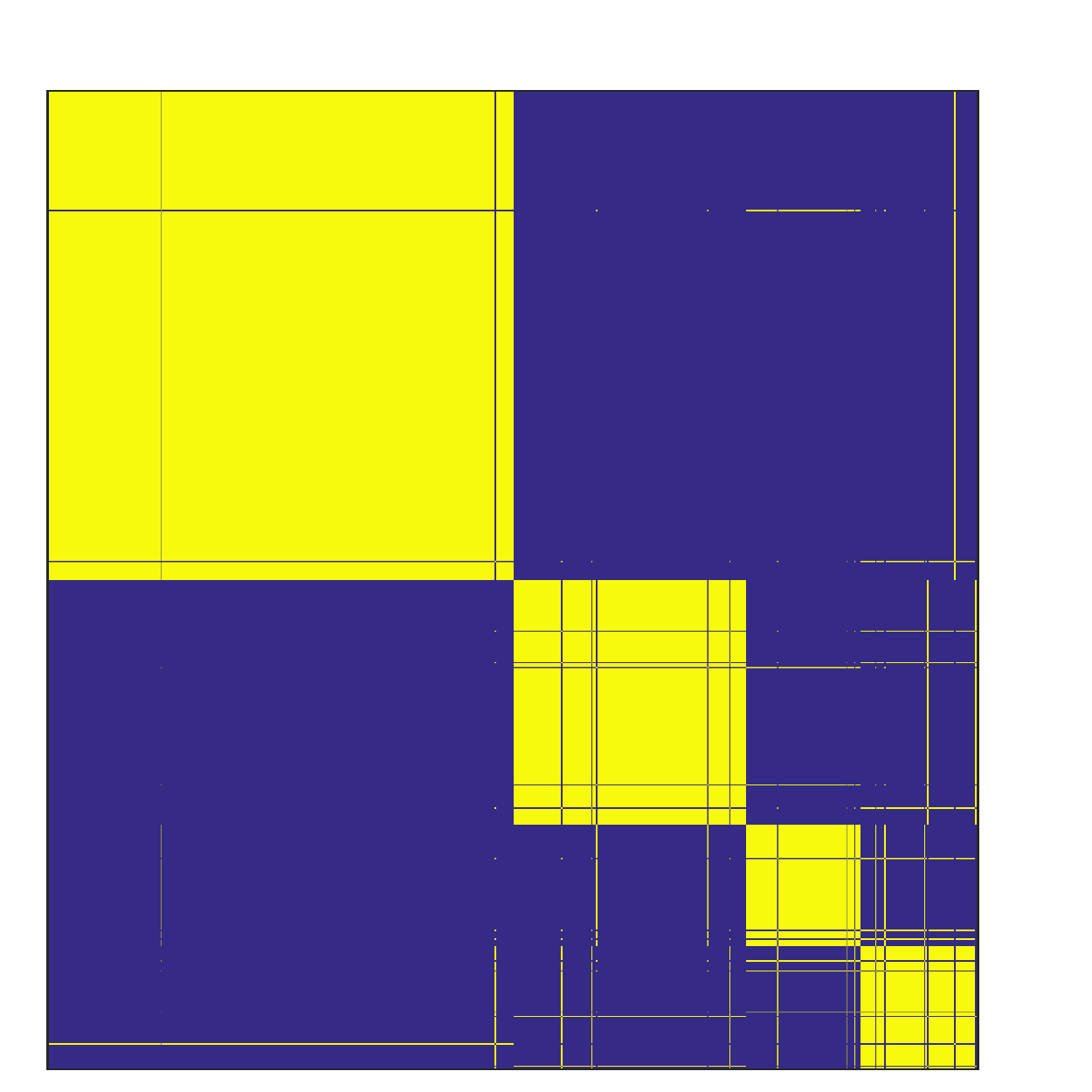}
        \caption{$\hat{z}$}
    \end{subfigure}
    \begin{subfigure}[b]{0.24\textwidth}
        \includegraphics[width=\textwidth]{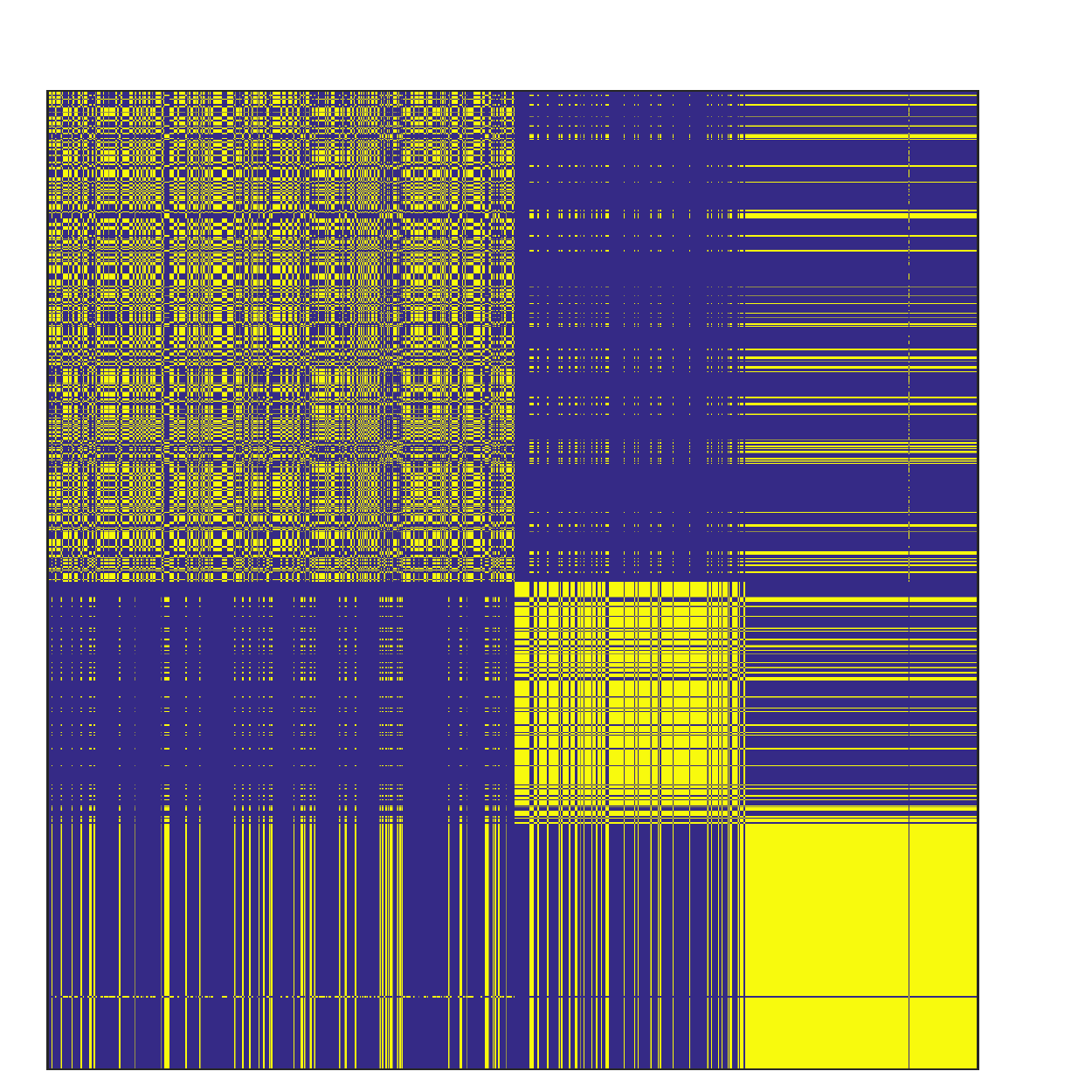}
        \caption{Spectral clusters}
    \end{subfigure}
    \caption{Community structured blockmodel \eqref{eq: community model} with $n=800$ nodes, $K=4$ classes, average degree $=18$, and unbalanced class sizes. (a) adjacency matrix $A$. (b) similarity matrix for $\hat{z}$, which resulted in 4 miss-classified nodes. (c) shows the similarity matrix for class estimates found by spectral clustering of $A$, which split the largest community, failed to find the smallest two communities, and resulted in $447$ miss-classified nodes.}\label{fig: community}
\end{figure}

\begin{figure} 
    \centering
    \begin{subfigure}[b]{0.24\textwidth}
        \includegraphics[width=\textwidth]{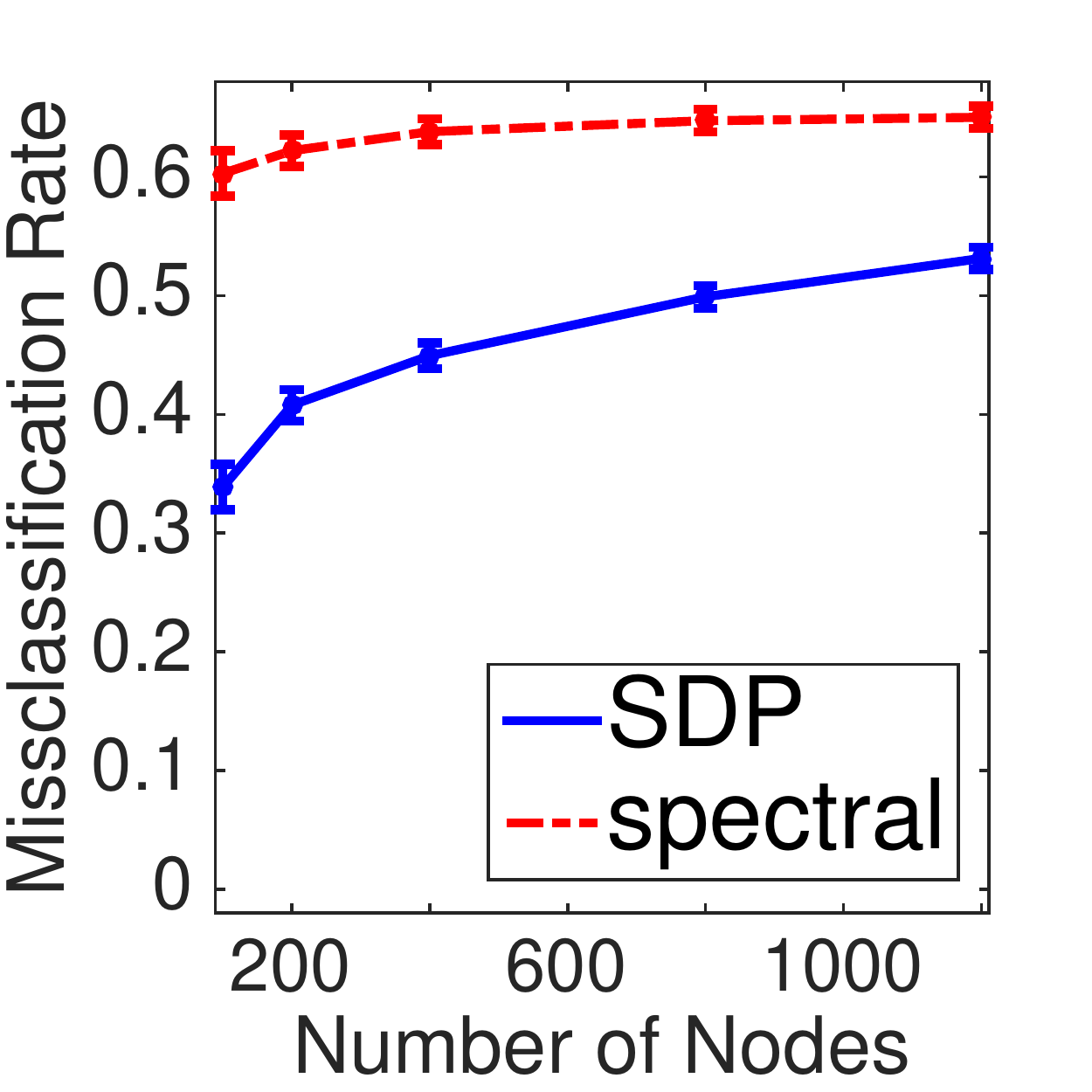}
        \caption{$n\rho = 5$ }
    \end{subfigure}
    \begin{subfigure}[b]{0.24\textwidth}
        \includegraphics[width=\textwidth]{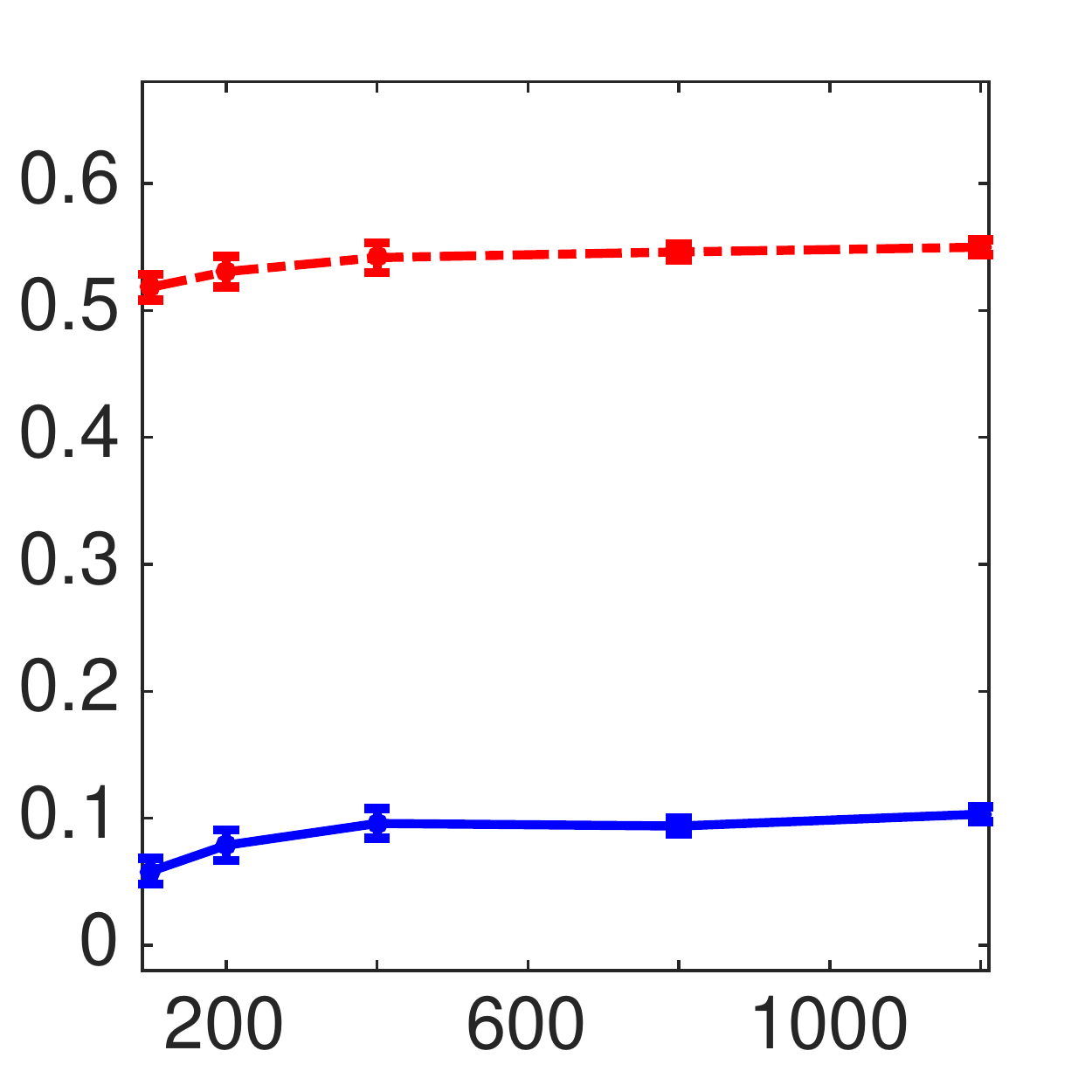}
        \caption{$n\rho = 10$}
    \end{subfigure}
    \begin{subfigure}[b]{0.24\textwidth}
        \includegraphics[width=\textwidth]{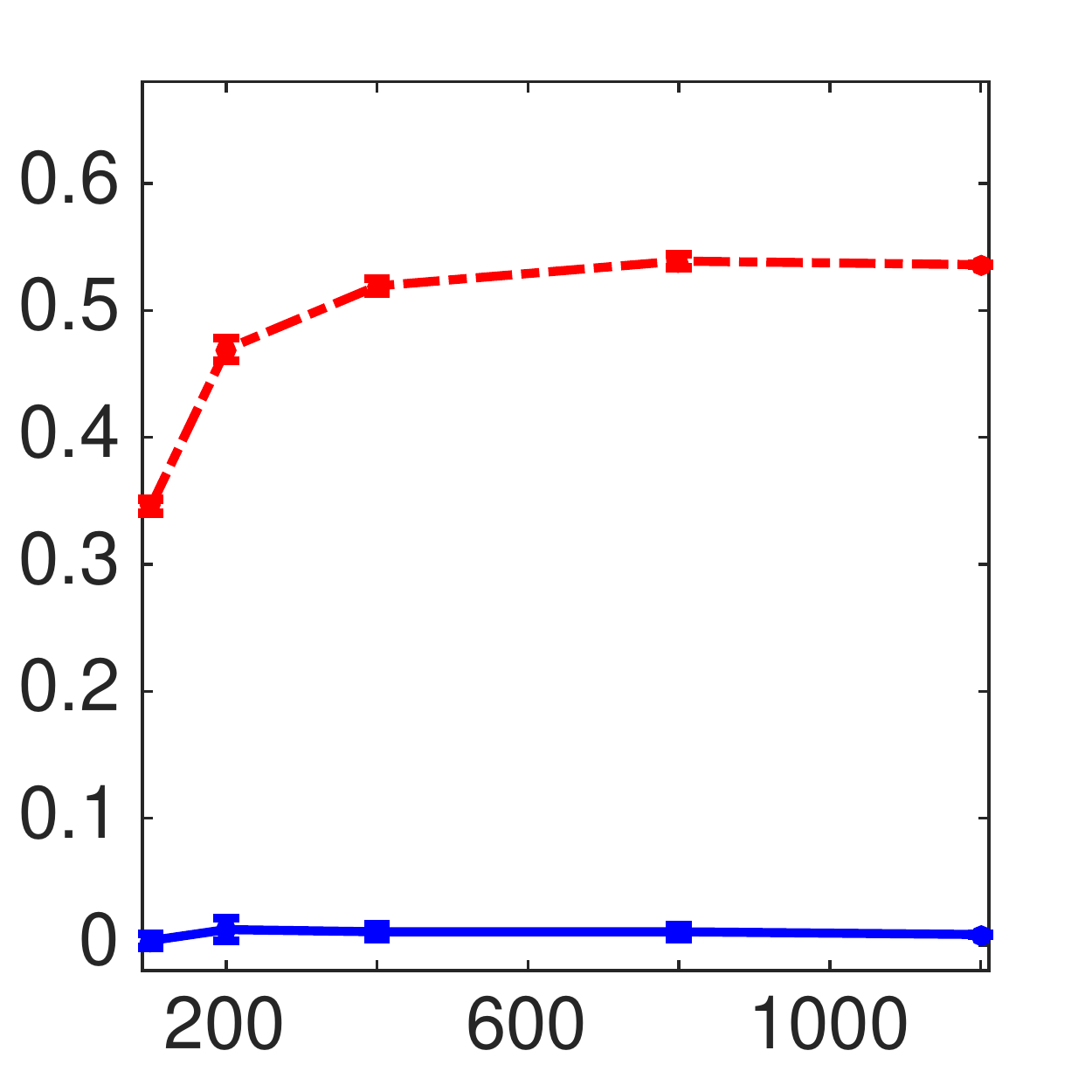}
        \caption{$n\rho = 20$}
    \end{subfigure}
    \begin{subfigure}[b]{0.24\textwidth}
        \includegraphics[width=\textwidth]{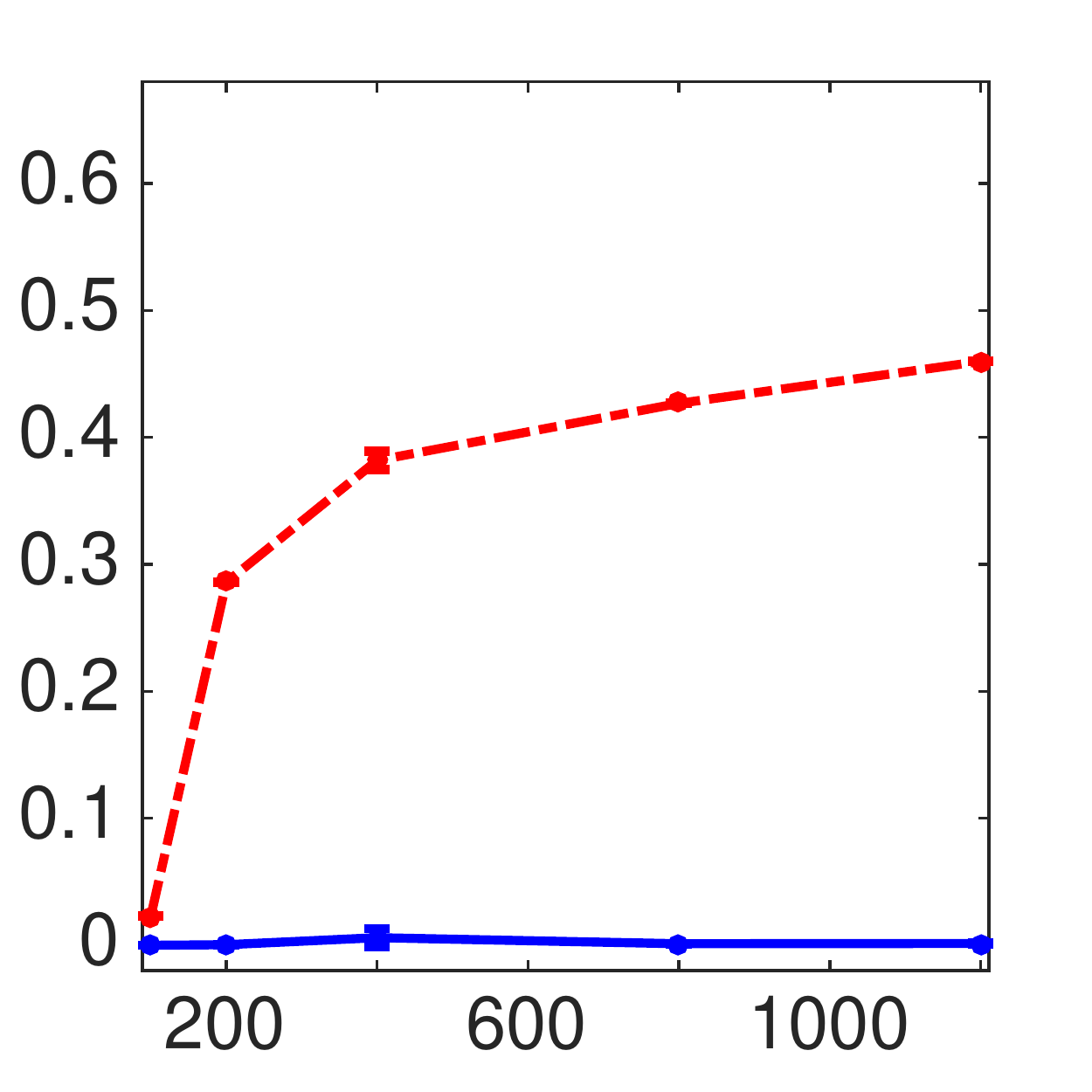}
        \caption{$n\rho = 30$}
    \end{subfigure}
    \caption{Misclassification rates for $\hat{z}$ found by semidefinite programming (blue line, solid) vs. spectral clustering of $A$ (red line, dotted), under the community structured blockmodel \eqref{eq: community model} with $K=4$, $n = \{100, 200, 400, 800, 1200\}$, and expected average degree $n\rho = \{5, 10, 20, 30\}$. 100 simulations per trial, standard errors shown.} %The semidefinite program always has lower error, which is roughly constant in $n$ for each fixed value of $n\rho \geq 10$. For spectral clustering of $A$, the error rate worses worsens with $n$, as exemplified by subplot (d). } 
	\label{fig: batch community}
\end{figure}

%Using this information, $Y^{t+1}$ can be 

%When $\gamma_0 > \gamma_1$, nodes in the same latent class are more likely to connect to each other than to nodes in others classes. In this case, existing semidefinite programming methods such as \cite{amini2014semidefinite, guedon2015community, montanari2015semidefinite} can be used to recover the latent classes. For all choices of $\gamma_0$ and $\gamma_1$, the semidefinite program \eqref{eq: sdp} can be applied, and by Corollary \ref{cor: sbm} and \ref{cor: sbm2}, $z$ can be estimated up to label switching with error decreasing with the average degree of $A$.

\subsection{Two Overlapping Groups of Communities} \label{sec: overlapping}

Let parameters $\gamma_0,\gamma_1 \in [0,1]$ satisfy $\gamma_0 > \gamma_1$, and let $K = k^2$ for some integer $k$, so that each $a \in [K]$ has a $k$-ary representation $(a_1,a_2) \in [k]^2$. Let $\theta$ equal
\begin{equation} \label{eq: overlapping}
	\theta_{ab} = \begin{cases} \gamma_0 & \textrm{ if } a_1 = b_1 \textrm{ or } a_2 = b_2 \\ \gamma_1 & \textrm{ otherwise}\end{cases} \qquad a,b\in [K].
\end{equation}
In this model, there exist two types of community structure, each type comprised of $k$ communities, where node belongs to one community of each type. Two nodes have a higher probability of connection if they share at least one community in common.

We observe that $\theta = \gamma_0 B_0 + \gamma_0 B_1 + \gamma_1 B_2$, where $B_0, B_1,$ and $B_2 \in \{0,1\}^{K \times K}$ are given by
\begin{align*}
B_0 &= I \\
B_1 &= (I - 11^T) \otimes I + I \otimes (I - 11^T) \\
B_2 &= (I - 11^T) \otimes (I - 11^T).
\end{align*}
By manual verification, $\{B_0,B_1,B_2\}$ can be seen to satisfy the requirements given in Definition \ref{def: association scheme} for an association scheme with $\ell = 2$.

For this model, it may be of interest to not only compute $\hat{z}$ and $\hat{\theta}^{\textrm{est}}$, which estimate $z$ and $\theta$ up to label switching, but to also estimate the $2k$ overlapping communities, which we will denote by $\mathcal{C}_1,\ldots,\mathcal{C}_{2k} \subset [n]$. It can be seen that $\mathcal{C}_1,\ldots,\mathcal{C}_{2k}$ are given by
\[\mathcal{C}_{l + (m-1)k} = \{i \in [n]: z_{im} = l\} \qquad l\in [k], m\in [2],\]
where each $z_i \in [K]$ has $k$-ary representation $(z_{i1},z_{i2})$. Equivalently, $\mathcal{C}_1,\ldots,\mathcal{C}_{2k}$ may also be defined as follows: Let $G$ denote a graph with $K$ vertices, and with edges induced by thresholding $\theta$ between $\gamma_0$ and $\gamma_1$:
\[ G_{ab} = \begin{cases} 1 & \textrm{ if } \theta_{ab} \geq (\gamma_0 + \gamma_1)/2 \\ 0 & \textrm{ otherwise.} \end{cases} \qquad a,b \in [K].\]
It can be seen that $G$ has $2k$ maximal cliques $C_1,\ldots,C_{2k}$, for which the communities $\mathcal{C}_1,\ldots,\mathcal{C}_{2k}$ satisfy
\begin{equation}\label{eq: S}
\mathcal{C}_\ell = \{i \in [n]: z_i \in C_{\pi(\ell)}\} \qquad \ell=1,\ldots,2k,
\end{equation}
for some permutation $\pi$ of $[2k]$.

Thus, to estimate $\mathcal{C}_1,\ldots,\mathcal{C}_{2k}$, we can construct $\hat{\theta}^{\textrm{est}} \in [0,1]^{K \times K}$ as given by \eqref{eq: hat theta est}, and estimate $G$ by 
\begin{equation}\label{eq: Ghat}
\hat{G}_{ab} = 1\left\{\hat{\theta}^{\textrm{est}}_{ab} \geq \frac{\gamma_0 + \gamma_1}{2}\right\} \qquad a,b \in [K],
\end{equation}
which has maximal cliques $\hat{C}_1,\ldots, \hat{C}_{m'} \subset [K]$ for some $m'$. Using the maximal cliques of $\hat{G}$, we can estimate overlapping communities $\hat{\mathcal{C}}_1,\ldots,\hat{\mathcal{C}}_{m'} \subset [n]$ by
\begin{equation}\label{eq: Shat}
	\hat{\mathcal{C}}_\ell = \{i \in [n]: \hat{z}_i \in \hat{C}_\ell\}, \qquad \ell=1,\ldots, m'.
\end{equation}

We remark that even in settings where the model \eqref{eq: overlapping} is not valid, by \eqref{eq: Ghat} and \eqref{eq: Shat} the subsets $\hat{\mathcal{C}}_1,\ldots,\hat{\mathcal{C}}_{m'}$ are still interpretable as overlapping subsets of densely connected nodes. 

Figure \ref{fig: overlapping}a shows an adjacency matrix $A$ generated by \eqref{eq: overlapping}, with $k=3$ and $K=9$. The pattern of $\theta$ is clearly visible in $A$. Figures \ref{fig: overlapping}b and \ref{fig: overlapping}c show the estimated communities $\hat{\mathcal{C}}_1,\ldots,\hat{\mathcal{C}}_6$ using \eqref{eq: Shat}. For comparison, Figure \ref{fig: overlapping}d shows $\hat{z}$ when the semidefinite program assumes \eqref{eq: community model} instead of \eqref{eq: overlapping}, and Figure \ref{fig: overlapping}e shows the estimate of $z$ under spectral clustering of $A$. In this instance, $\hat{\mathcal{C}}_1,\ldots,\hat{\mathcal{C}}_6$ accurately estimate the true communities $\mathcal{C}_1,\ldots,\mathcal{C}_6$, with a misclassification rate of $0.04$. This is consistent with Theorem \ref{th: sbm}, which predicts that $\hat{P}$ will be nearly block-structured, implying that the subsequent steps of estimating $\hat{z}, \hat{\theta}^{\textrm{est}},$ and $\hat{G}$ will succeed as well. In contrast, the two alternative methods give poor estimates for $z$, with misclassification rates of $0.50$ and $0.23$, respectively.

Figure \ref{fig: batch overlapping} shows the average misclassification rate for $\hat{\mathcal{C}}_1,\ldots,\hat{\mathcal{C}}_6$, in simulations over a range of values for the network size $n$ and average degree $n\rho$. For comparison, the misclassification rate for spectral clustering of $A$ (which estimates $z$ instead of $\mathcal{C}_1,\ldots,\mathcal{C}_6$) is shown as well. The misclassification rate of the semidefinite program is not quite constant in $n$ for fixed $n\rho$, suggesting that the asymptotic results of Theorem \ref{th: sbm} may require larger $n$ for this model compared to the results shown in Figure \ref{fig: batch community}. However, as $n\rho$ increases the semidefinite program estimates the overlapping communities (and hence $z$ as well) with much better accuracy than spectral clustering of $A$, which shows little improvement with increasing $n\rho$.

\begin{figure}
    \centering
    \begin{subfigure}[b]{0.19\textwidth}
        \includegraphics[width=\textwidth]{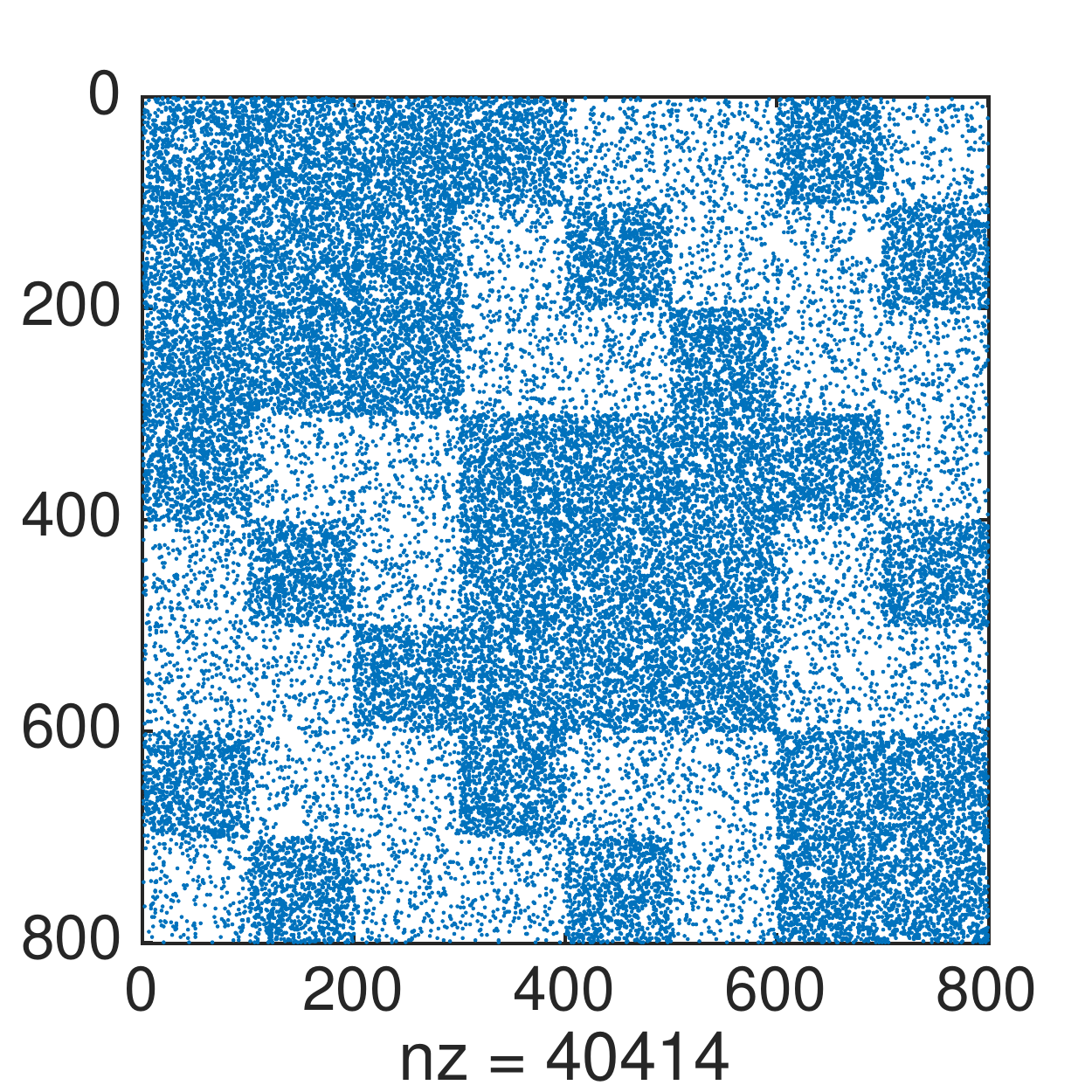}
		\caption{Adjacency matrix $A$}
    \end{subfigure}
   \begin{subfigure}[b]{0.19\textwidth}
       \includegraphics[width=\textwidth]{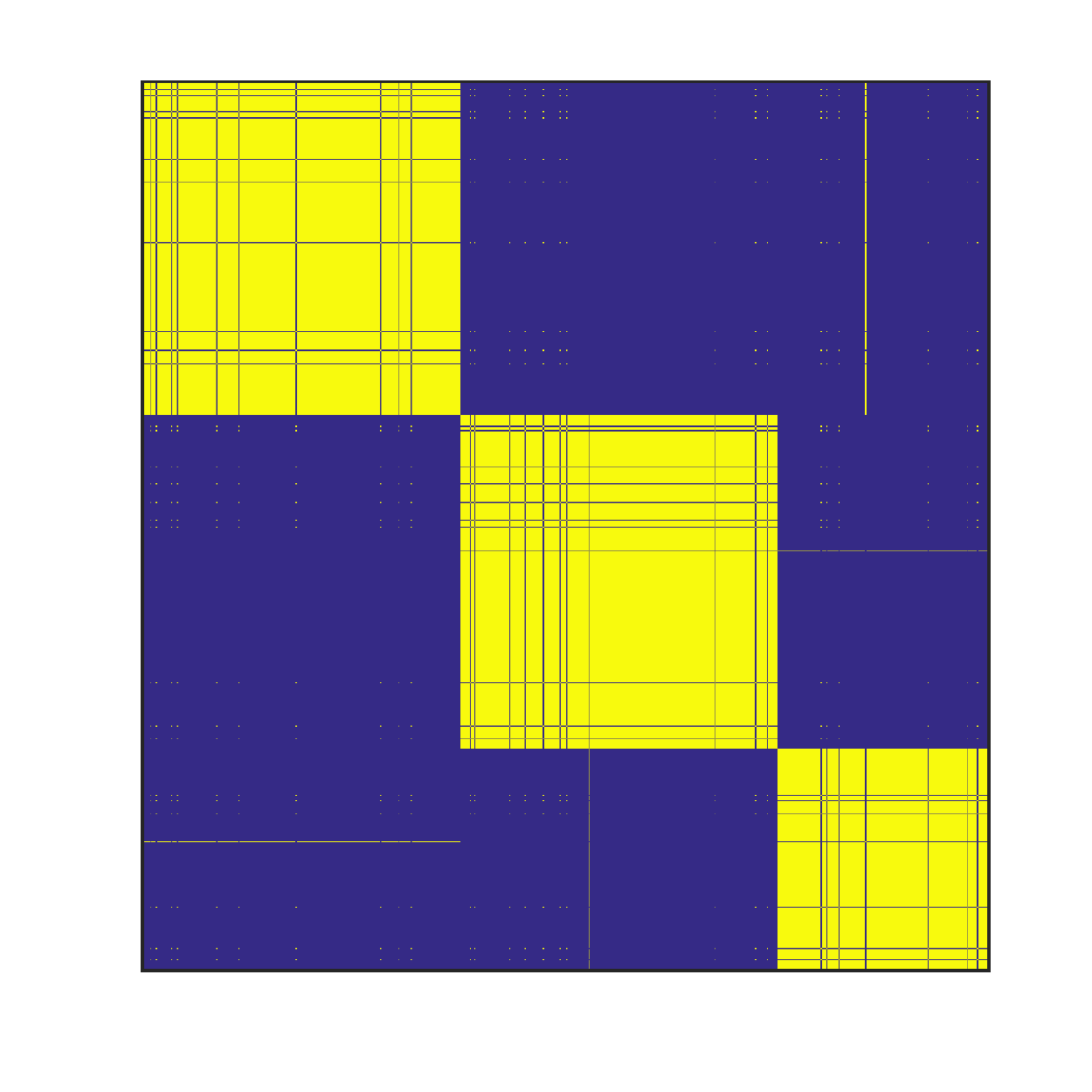}
		\caption{$\hat{\mathcal{C}}_1,\hat{\mathcal{C}}_2,\hat{\mathcal{C}}_3$ \textcolor{white}{1 2 3 4}}
%        \caption{Estimated $\hat{Z}^{(0)}$ from SDP {\textcolor{white}{a a a a}}}
   \end{subfigure}
   \begin{subfigure}[b]{0.19\textwidth}
       \includegraphics[width=\textwidth]{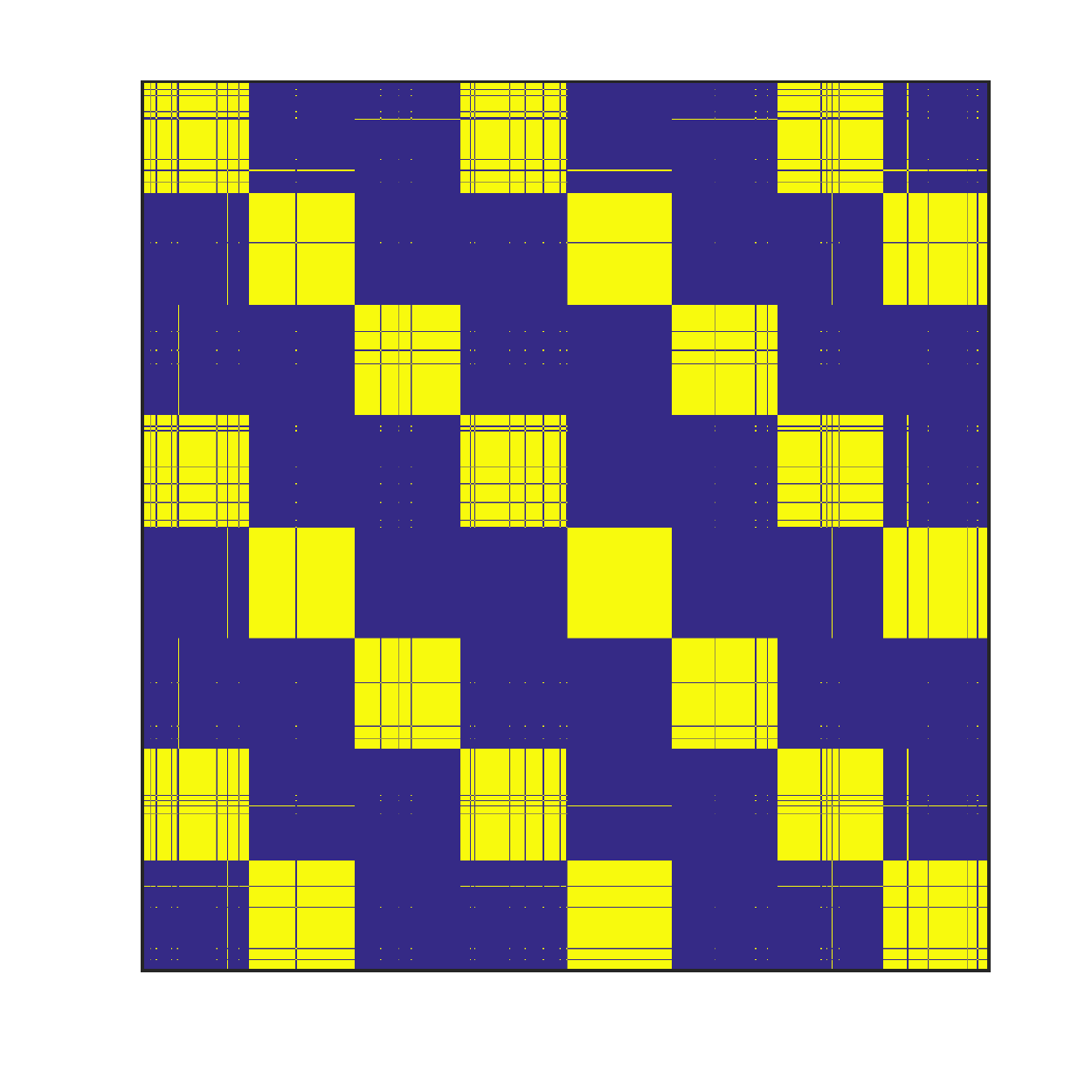}
		\caption{$\hat{\mathcal{C}}_4,\hat{\mathcal{C}}_5,\hat{\mathcal{C}}_6$ \textcolor{white}{1 2 3 4}}
%        \caption{Hierarchical Spectral Clust. of $A$}
   \end{subfigure}
   \begin{subfigure}[b]{0.19\textwidth}
       \includegraphics[width=\textwidth]{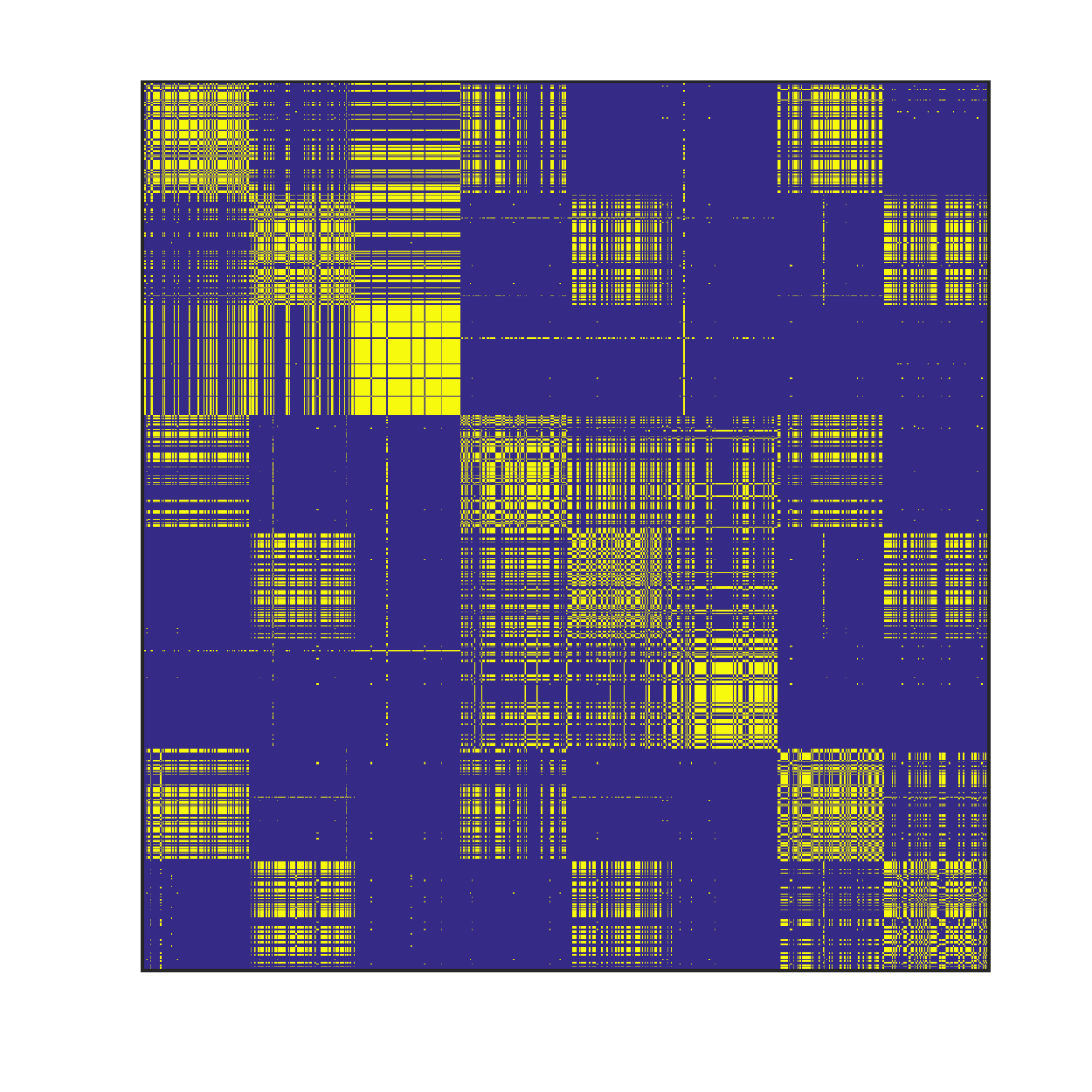}
		\caption{Nonoverlapping model}
%        \caption{Spectral Clustering of $A$}
   \end{subfigure}
   \begin{subfigure}[b]{0.19\textwidth}
       \includegraphics[width=\textwidth]{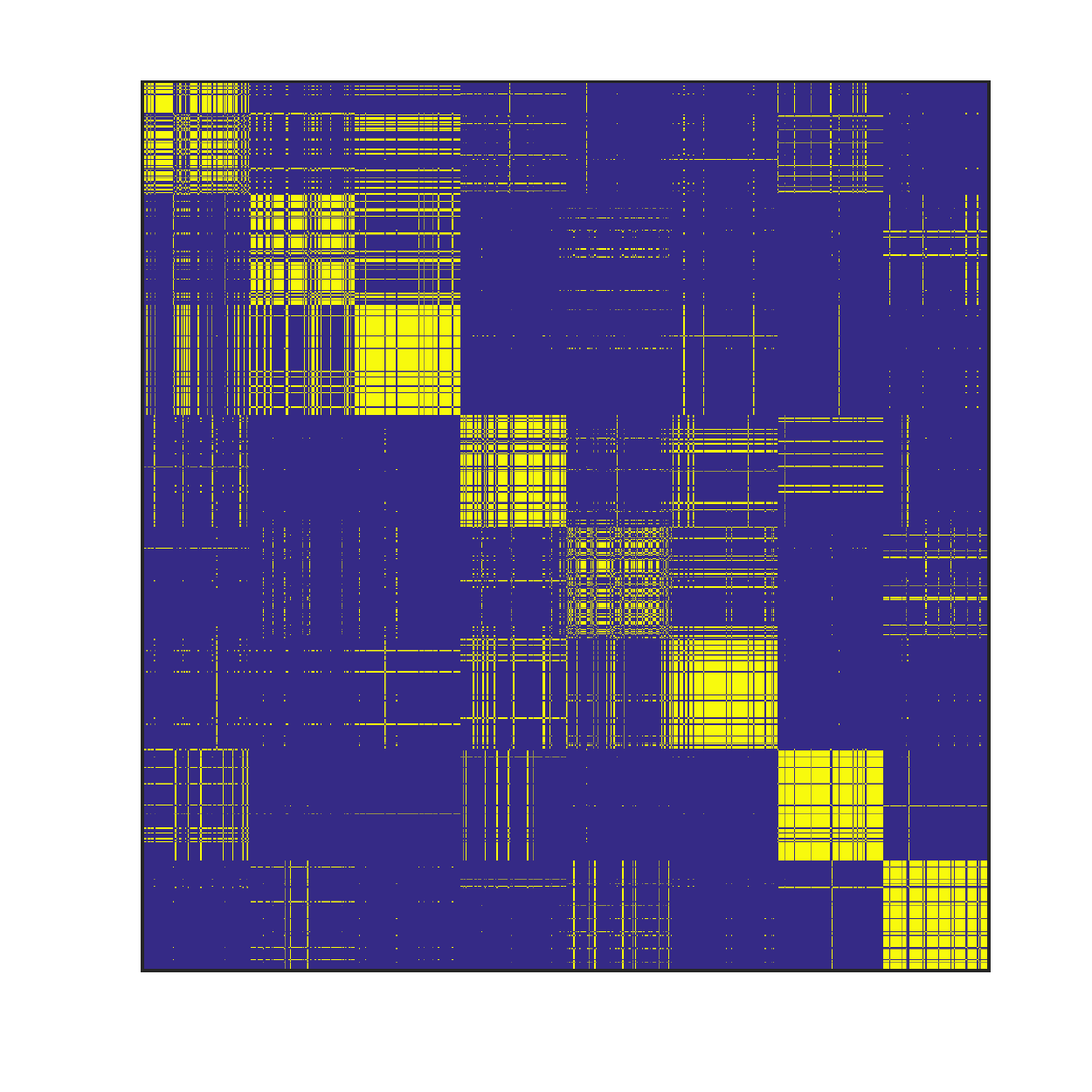}
		\caption{Spectral clustering of $A$}
%        \caption{Spectral Clustering of $A$}
   \end{subfigure}
    \caption{Overlapping community model \eqref{eq: overlapping}, with $n=800$, $n\rho = 50$, and $6$ overlapping communities $\mathcal{C}_1,\ldots,\mathcal{C}_6$. (a) adjacency matrix $A$. (b) similarity matrix for $\hat{\mathcal{C}}_1,\hat{\mathcal{C}}_2,\hat{\mathcal{C}}_3$. (c) similarity matrix for  $\hat{\mathcal{C}}_4,\hat{\mathcal{C}}_5,\hat{\mathcal{C}}_6$. (d) similarity matrix for $\hat{z}$, but assuming non-overlapping model \eqref{eq: community model}. (e) similarity matrix using spectral clustering of $A$. $\hat{\mathcal{C}}_1,\ldots,\hat{\mathcal{C}}_6$ had $32$ errors, while $\hat{z}$ using the non-overlapping model \eqref{eq: community model} had $397$ errors, and spectral clustering of $A$ had $187$ errors.} \label{fig: overlapping}
\end{figure}

\begin{figure} 
    \centering
    \begin{subfigure}[b]{0.24\textwidth}
        \includegraphics[width=\textwidth]{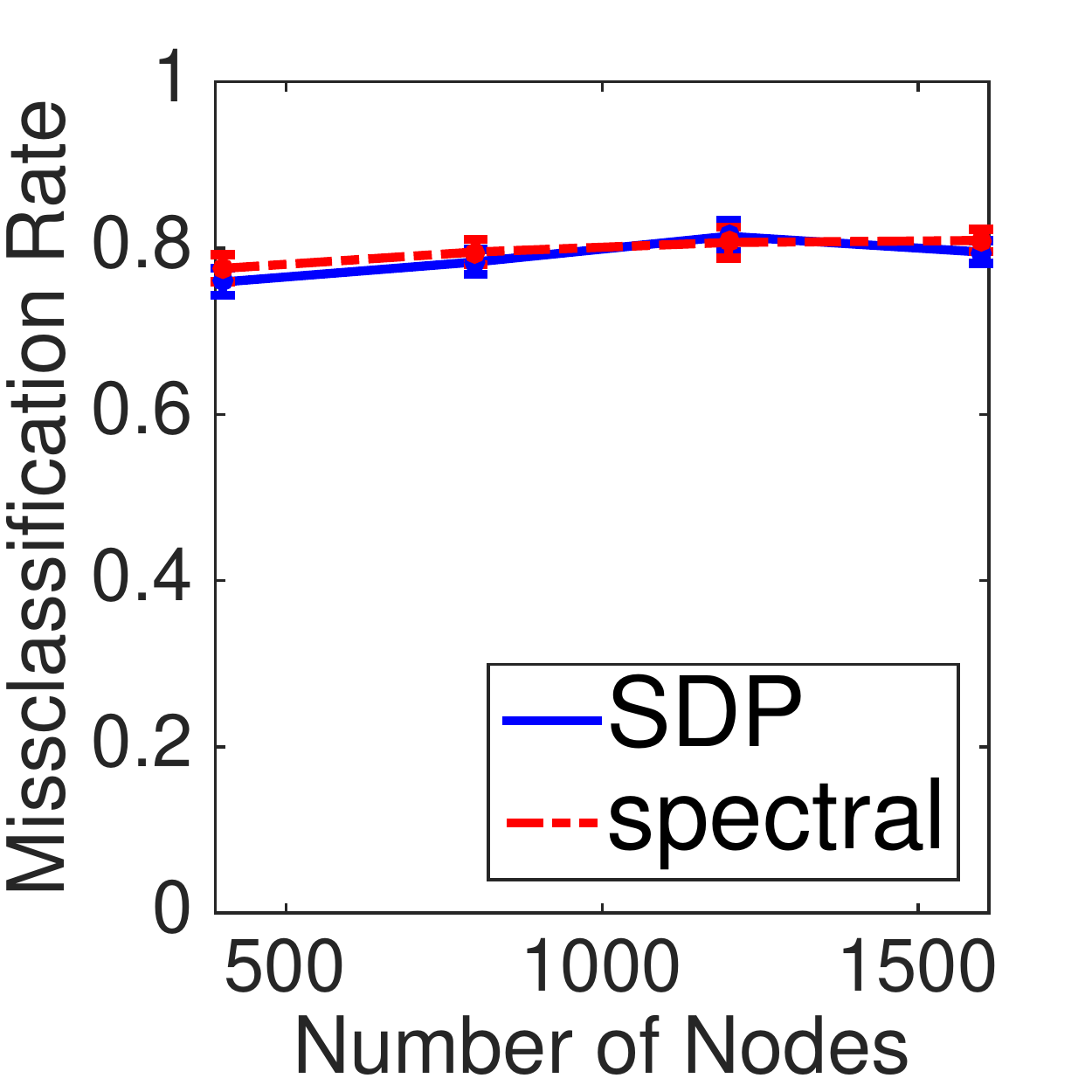}
        \caption{$n\rho = 30$ }
    \end{subfigure}
    \begin{subfigure}[b]{0.24\textwidth}
        \includegraphics[width=\textwidth]{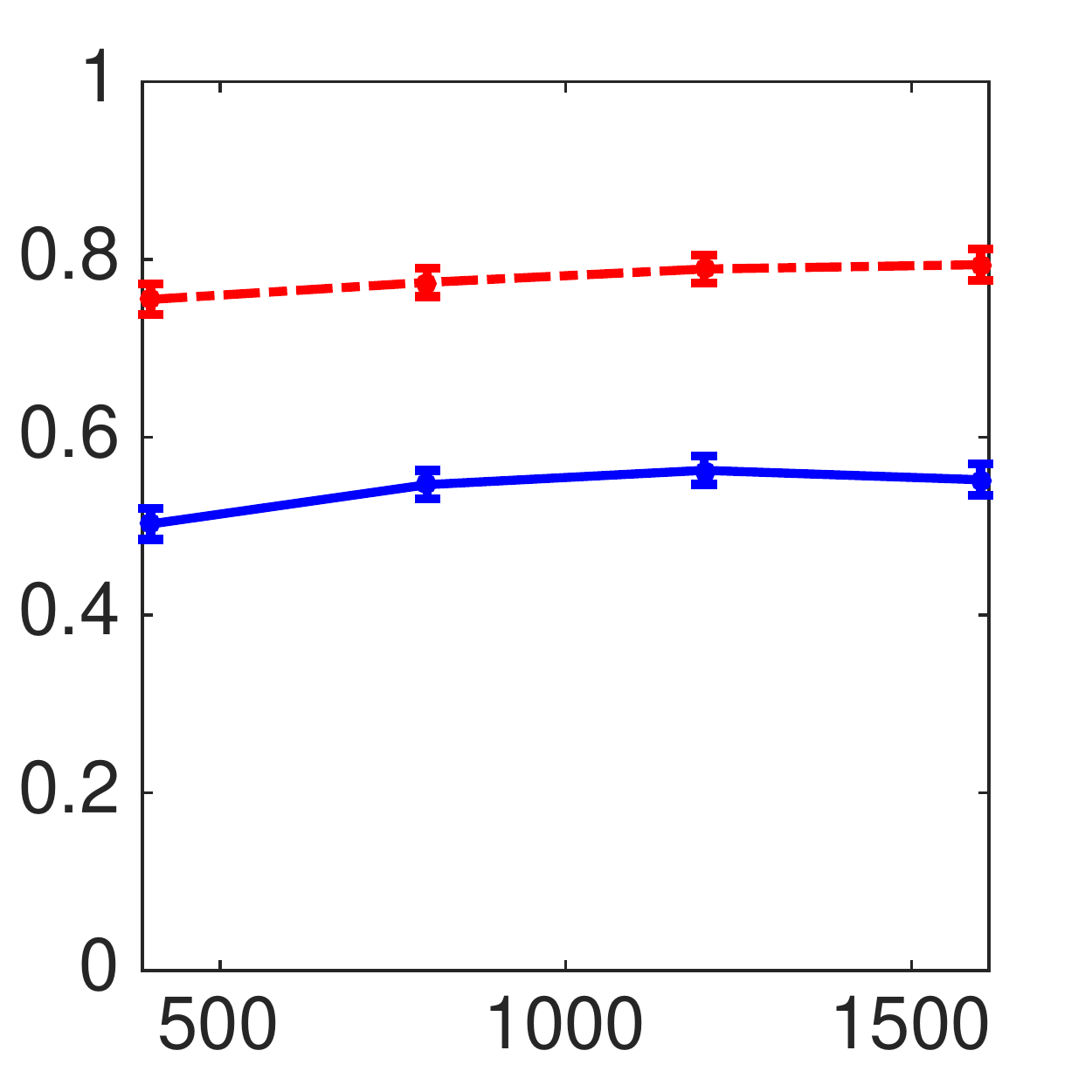}
        \caption{$n\rho = 40$}
    \end{subfigure}
    \begin{subfigure}[b]{0.24\textwidth}
        \includegraphics[width=\textwidth]{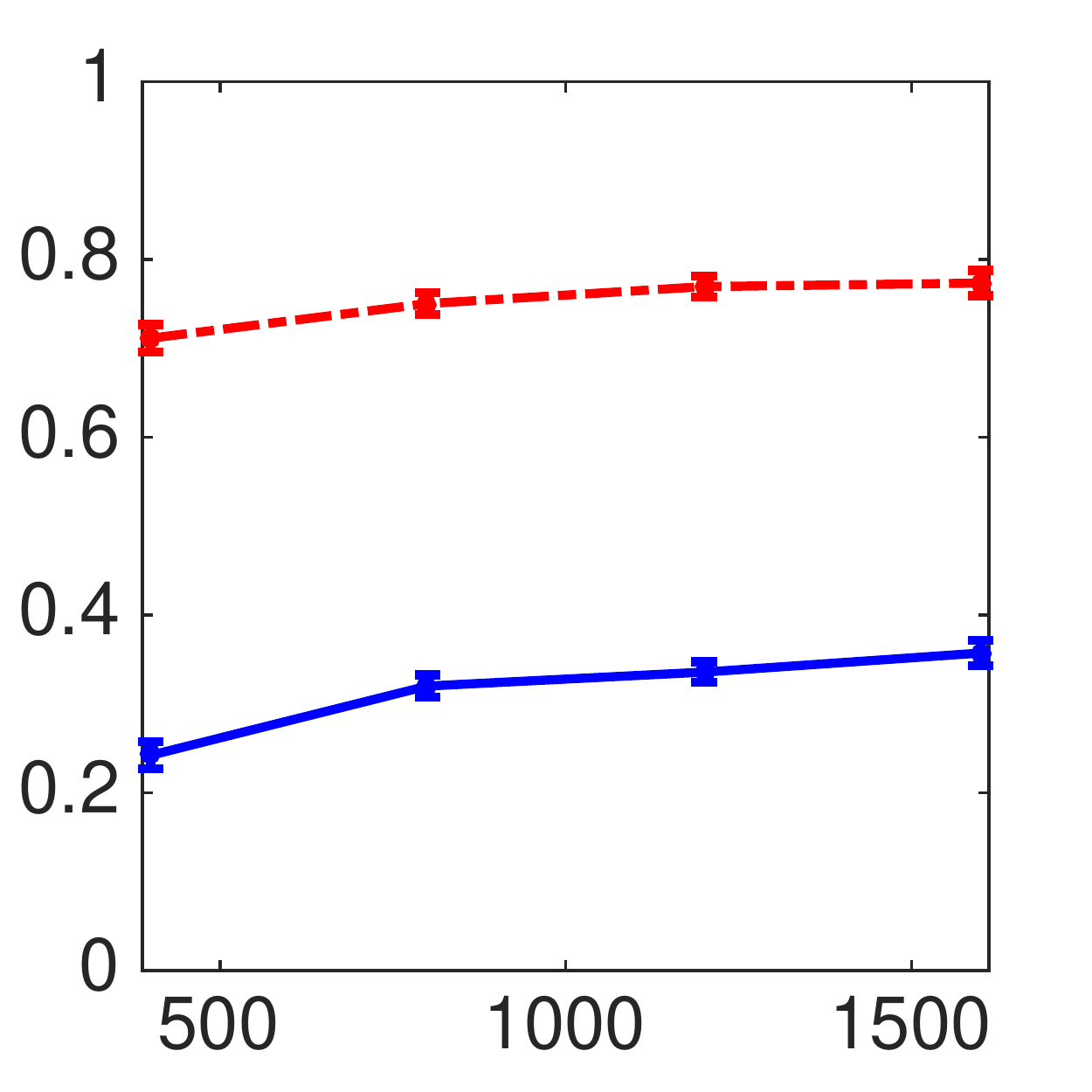}
        \caption{$n\rho = 50$}
    \end{subfigure}
    \begin{subfigure}[b]{0.24\textwidth}
        \includegraphics[width=\textwidth]{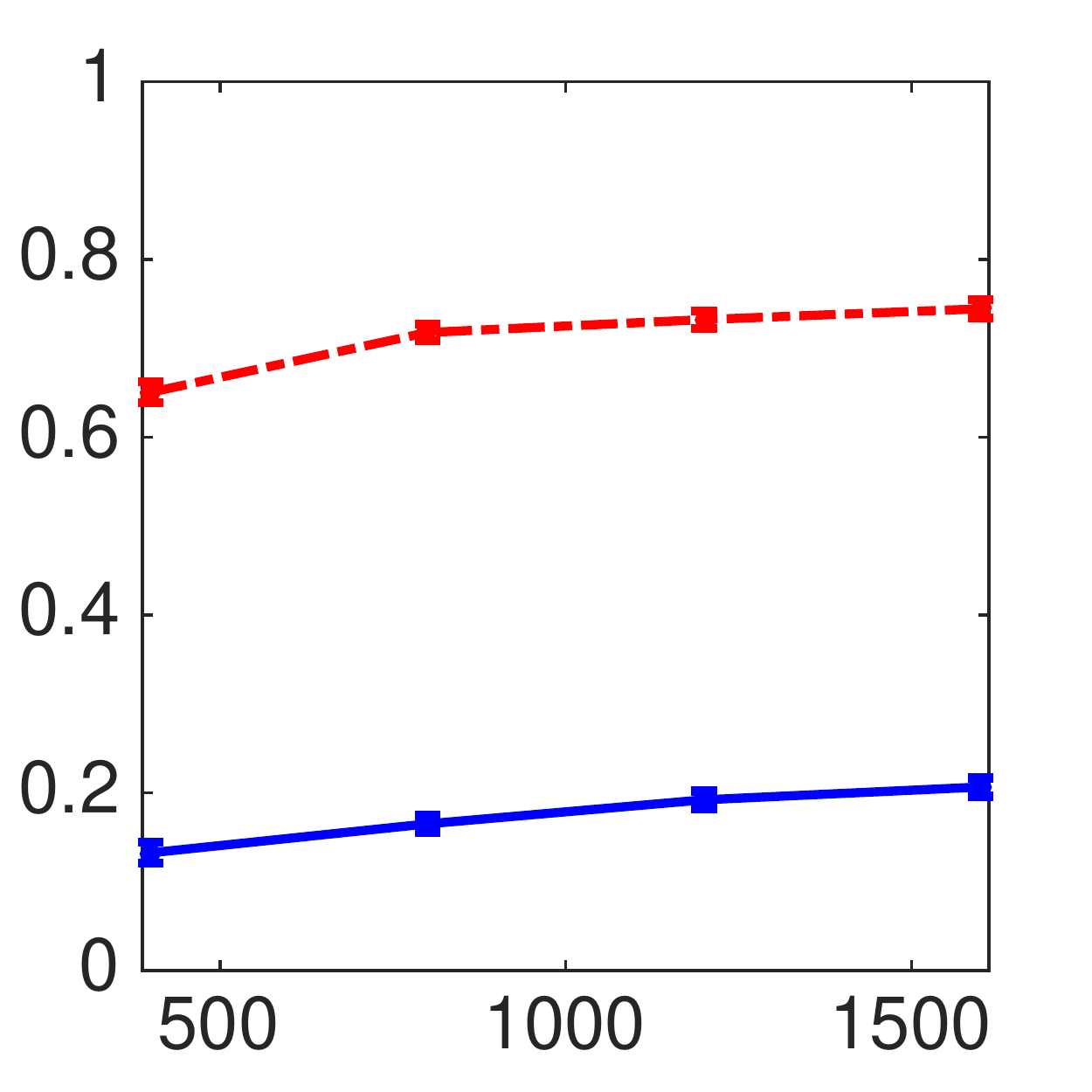}
        \caption{$n\rho = 60$}
    \end{subfigure}
    \caption{Misclassification rates for $\hat{\mathcal{C}}_1,\ldots,\hat{\mathcal{C}}_6$ found by semidefinite programming (blue line, solid), and for estimates of $z$ found by spectral clustering of $A$ (red line, dotted), under the overlapping community model \eqref{eq: overlapping} with $K=9$, $k=3$, $n = \{400, 800, 1200, 1600\}$, expected average degree $n\rho = \{30, 40, 50, 60\}$, and unbalanced class sizes. 100 simulations per trial, standard errors shown.} \label{fig: batch overlapping}
\end{figure}

\subsection{Latent Space Models} \label{sec: hoff}

We consider a latent space model, reminiscent of \cite{hoff2002latent}, in which each node is assigned a latent coordinate vector $y_1,\ldots,y_n \in [0,1]^d$. Conditional on $\{y_i\}_{i=1}^n$, each dyad is independent Bernoulli with log odds given by
\begin{equation}\label{eq: hoff}
\textrm{log odds }A_{ij} = -\|y_i - y_j\|/\sigma,
\end{equation}
where $\sigma \geq 0$ is a bandwidth parameter. In general, \eqref{eq: hoff} is not equivalent to a blockmodel.

Let $D\in \mathbb{R}^{n \times n}$ denote the matrix of squared distances between the latent coordinates, given by
\begin{align*}
    D_{ij} &= \|y_i - y_j\|^2.
\end{align*}
It is known that the first $d$ eigencoordinates of $(I - 11^T/n)D(I- 11^T/n)$ will recover $y_1,\ldots,y_n$ up to a unitary transformation.

To estimate $D$, we will approximate \eqref{eq: hoff} by a blockmodel with $K = (2k)^d$ classes, where each class represents a coordinate in $\mathbb{R}^d$, so that $\hat{\theta}$ equals
\begin{equation}\label{eq: hoff sbm}
\operatorname{logit} \hat{\theta}_{ab} = -\delta(\gamma_a,\gamma_b)/\sigma,
\end{equation}
where $\gamma_1,\ldots,\gamma_K \in \mathbb{R}^d$, and $\delta$ is a distance metric. In order for $\hat{\theta}$ to belong to an association scheme, we choose $\gamma_1,\ldots,\gamma_K$ extending beyond $[0,1]^d$ to form a grid in $[0,2]^d$, and choose $\delta$ to be a toric distance. That is, given $a \in [K]$, let $(a_1,\ldots,a_d)$ denote its $(2k)$-ary representation, and let $\gamma_1,\ldots,\gamma_K$ be given by
\[ \gamma_{a} = \left(\frac{a_1}{k}, \ldots, \frac{a_d}{k}\right), \qquad a \in [K],\]
and let $\delta:[0,2]^d \times [0,2]^d \mapsto \mathbb{R}$ equal the distance on a torus of circumference $2$,
\[ \delta(\gamma_a, \gamma_b) = \left[ \sum_{j=1}^d \min( |\gamma_{aj} - \gamma_{bj}|, 2 - |\gamma_{aj} - \gamma_{bj}|)^2\right]^{1/2}, \]
where $\gamma_{aj}$ is the $j$th element of the vector $\gamma_a$. Since $\hat{\theta}_{ab}$ depends on $\gamma_a$ and $\gamma_b$ only through their element-wise differences $|\gamma_{aj} - \gamma_{bj}|$, for $j=1,\ldots,d$, it follows that $\hat{\theta}$ can be written as a weighted sum of matrices
\[ \hat{\theta} = \sum_{j_1=0}^{2k-1} \cdots \sum_{j_d=0}^{2k-1} \gamma_{j_1 \cdots j_d} C^{(j_1)} \otimes \cdots \otimes C^{(j_d)},\]
where $C^{(j)} \in \{0,1\}^{2k \times 2k}$ for $j=0,\ldots,2k-1$ is the circulant matrix given by
\[ C^{(j)}_{ab} = \begin{cases} 1 & \textrm{ if } \min(|a-b|, 2k - |a-b|) = j \\ 0 & \textrm{ otherwise.} \end{cases} \]
By manual inspection, it can be seen that $\mathcal{C} = \{C^{(0)}, \ldots, C^{(2k-1)}\}$ satisfies the requirements of an association scheme, and hence that also $\mathcal{C}\otimes \cdots \otimes \mathcal{C}$ is an association scheme as well, with $\ell \leq (2k)
^d = K$. (See \cite[Chapters 1-3]{bailey2004association} for a complete treatment.)

%Let us complete the description of the estimation procedure, before discussing the association scheme. 

Let $\hat{X}$ denote the solution to the semidefinite program \eqref{eq: sdp} with $\hat{\theta}$ given by \eqref{eq: hoff sbm}. To estimate $D$, let $\hat{D} \in \mathbb{R}^{n \times n}$ equal the MAP-style estimate under $\hat{X}$,
\begin{equation}\label{eq: Dhat}
\hat{D}_{ij} = \arg \max_{x \in \mathbb{R}} \sum_{a=1}^K \sum_{b=1}^K \hat{X}_{ab}^{(ij)} 1\{\delta(\gamma_a,\gamma_b)^2 = x\},
\end{equation}
and let $\tilde{D} \in \mathbb{R}^{n \times n}$ denote the randomized estimate
\begin{equation}\label{eq: Dtilde}
 \tilde{D}_{ij} = \delta(\gamma_a,\gamma_b)^2 \textrm{ with probability } X^{(ij)}_{ab}.
 \end{equation}
Given $\hat{D}$, we will take the first $d$ eigencoordinates of $(I - 11^T/n)\hat{D}(I - 11^T/n)$ to estimate $y_1,\ldots,y_n$ (or similarly using $\tilde{D}$).

Corollary \ref{cor: th2} bounds the error between the randomized estimate $\tilde{D}$ and the true distances $D$, relative to rounding the coordinates $\{y_i\}_{i=1}^n$ to their closest points on a grid over $[0,1]^d$:
\begin{corollary}\label{cor: th2}
Let $L:\mathbb{R}\times \mathbb{R} \mapsto \mathbb{R}$ denote the loss function
\[L(x,x') = KL\left(\frac{e^{-x/\sigma}}{1+e^{-x/\sigma}}, \frac{e^{-x'/\sigma}}{1+e^{-x'/\sigma}}\right),\]
and let $\eta_1,\ldots,\eta_{k^d} \in \mathbb{R}$ form a grid on $[0,1]^d$,
\[ \eta_{a} = \left(\frac{a_1}{k}, \ldots, \frac{a_d}{k}\right), \qquad a \in [k^d],\]
where $(a_1,\ldots,a_d)$ denotes the $k$-ary representation of $a \in [k^d]$.
Then
\begin{align*}
\frac{1}{n^2\rho} \sum_{i=1}^n \sum_{j=1}^n L\left(\sqrt{D_{ij}},\sqrt{\tilde{D}_{ij}}\right) &\leq \min_{z \in [k^d]^n} \frac{1}{n^2\rho} \sum_{i=1}^n \sum_{j=1}^n L\left(\sqrt{D_{ij}}, \|\eta_{z_i} - \eta_{z_j}\|\right) \\
& \hskip.5cm {} + O_P\left(\frac{1}{\sqrt{n\rho}}\right)
\end{align*}
\end{corollary}
\begin{proof}
It holds that
\begin{align*}
\frac{1}{n^2\rho} \sum_{i=1}^n \sum_{j=1}^n L\left(\sqrt{D_{ij}},\sqrt{\tilde{D}_{ij}}\right) &\leq \min_{z \in [K]^n} \frac{1}{n^2\rho} \sum_{i=1}^n \sum_{j=1}^n L\left(\sqrt{D_{ij}}, \delta(\gamma_{z_i},\gamma_{z_j})\right) + O_P\left(\frac{1}{\sqrt{n\rho}}\right) \\
& \leq \min_{z \in [k^d]^n} \frac{1}{n^2\rho} \sum_{i=1}^n \sum_{j=1}^n L\left(\sqrt{D_{ij}}, \delta(\eta_{z_i},\eta_{z_j})\right) + O_P\left(\frac{1}{\sqrt{n\rho}}\right) \\
& \leq \min_{z \in [k^d]^n} \frac{1}{n^2\rho} \sum_{i=1}^n \sum_{j=1}^n L\left(\sqrt{D_{ij}}, \|\eta_{z_i} - \eta_{z_j}\|\right) + O_P\left(\frac{1}{\sqrt{n\rho}}\right),
\end{align*}
where the first inequality holds by Theorem \ref{th: Ptilde} and definition of $L$; the second inequality holds because the minimization is over a strictly smaller set than in the previous line; and the last inequality holds because $\delta(x,x') = \|x - x'\|$ for all $(x,x') \in [0,1]^d \times [0,1]^d$.
\end{proof}

%To show that $\hat{\theta}$ belongs to an association scheme, we first consider the case where $d=2$, and observe that 

%, where $\hat{\theta}$ equals
%\begin{align*}
%\hat{\theta}_{ab} & = -\delta(\gamma_a,\gamma_b)/\sigma \\
%\delta(\gamma_a,\gamma_b) & = \left[ \sum_{l=1}^d \min( |\gamma_{al} - \gamma_{bl}|, 2 - |\gamma_{al} - \gamma_{bl}|)
%\[ \operatorname{logit} \hat{\theta}_{ab} = - \delta(\gamma_a,\gamma_b)/\sigma,\]

Figure \ref{fig: hoff} shows latent coordinates $y_1,\ldots,y_n \in \mathbb{R}^2$ arranged in a circle, from which an adjacency matrix was generated by \eqref{eq: hoff} with $n=500$ and $n\rho = 20$. The figure also shows the estimated coordinates $\hat{y}$ derived from the randomized $\tilde{D}$ given by \eqref{eq: Dtilde}, from the MAP-style $\hat{D}$ given by \eqref{eq: Dhat}, and by applying the USVT method of \cite{chatterjee2015matrix} to the adjacency matrix $A$, in which a spectral estimate of $P$ is constructed from $A$, and then inverted to form an estimate of $D$. In this instance, both $\tilde{D}$ and $\hat{D}$ yield estimates $\hat{y}$ that are similar, and are substantially more accurate than the USVT approach which failed due to the sparsity of $A$. Figure \ref{fig: hoff2} shows a different configuration for $y_1,\ldots,y_n$, with similar results. 

Figure \ref{fig: batch hoff} shows the average simulated estimation accuracy using $\hat{D}$, for a range of values for the network size $n$ and average degree $n\rho$. For comparison, the performance of the spectral USVT method is shown as well. We see that the estimation error for $\hat{D}$ is near-constant when the average degree is fixed. In contrast, the estimation error for the spectral USVT method worsens with the sparsity of $A$. This is exemplified by subplot (d), in which the USVT method performs well for small networks but degrades severely as $n$ increases.

\begin{figure}
    \centering
    \begin{subfigure}[b]{0.24\textwidth}
        \includegraphics[width=\textwidth]{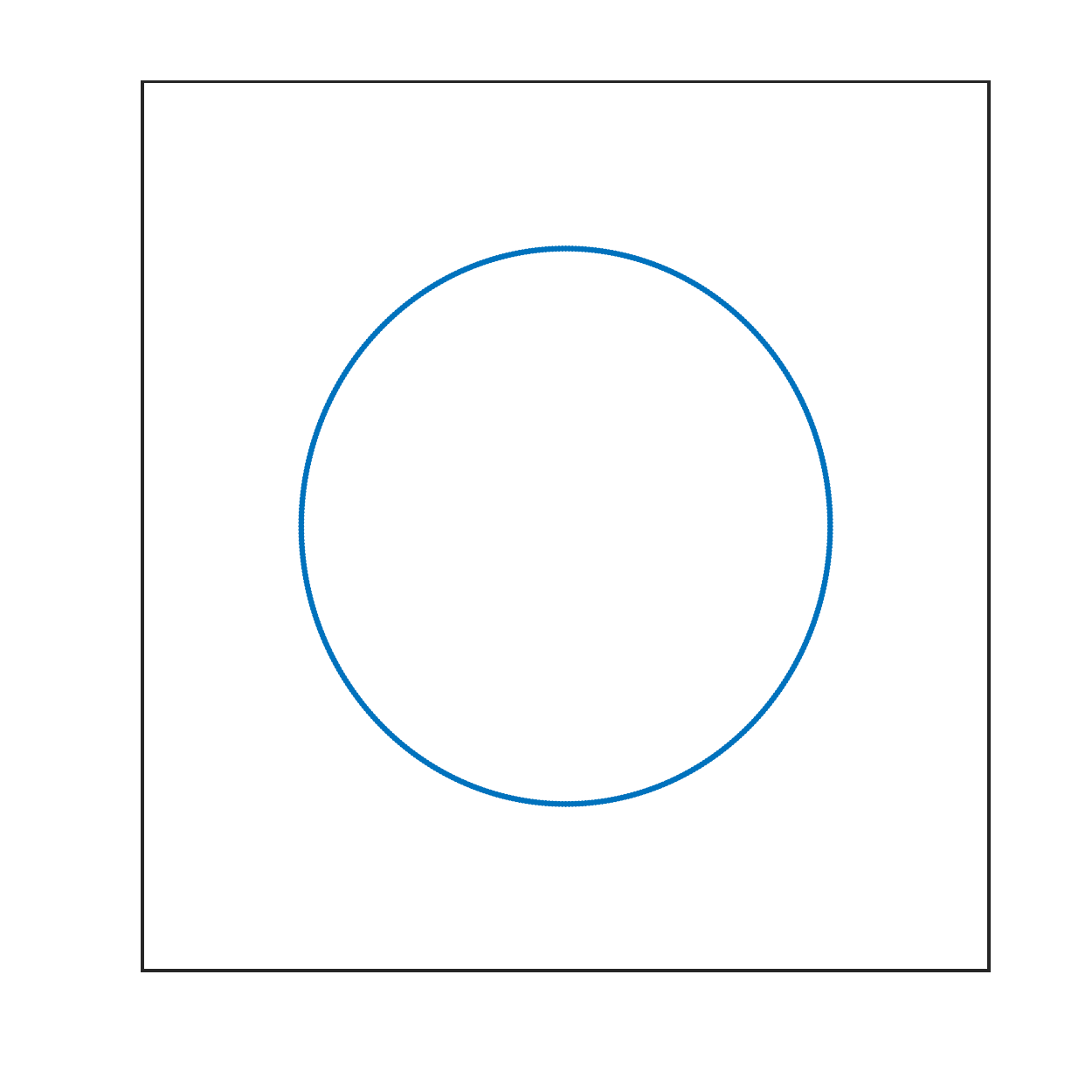}
        \caption{$y$ }
    \end{subfigure}
    \begin{subfigure}[b]{0.24\textwidth}
        \includegraphics[width=\textwidth]{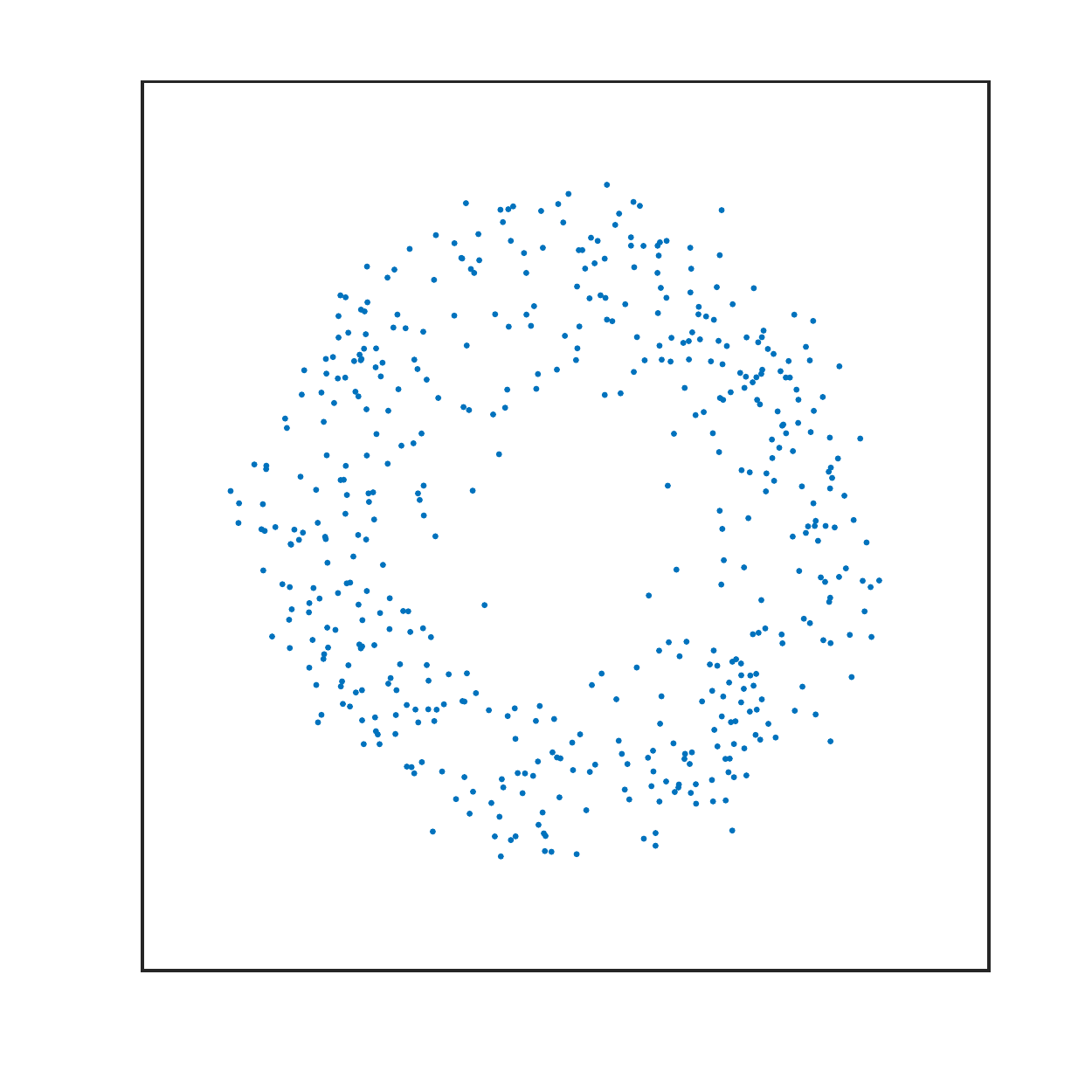}
        \caption{$\hat{y}$ from $\tilde{D}$}
    \end{subfigure}	
    \begin{subfigure}[b]{0.24\textwidth}
        \includegraphics[width=\textwidth]{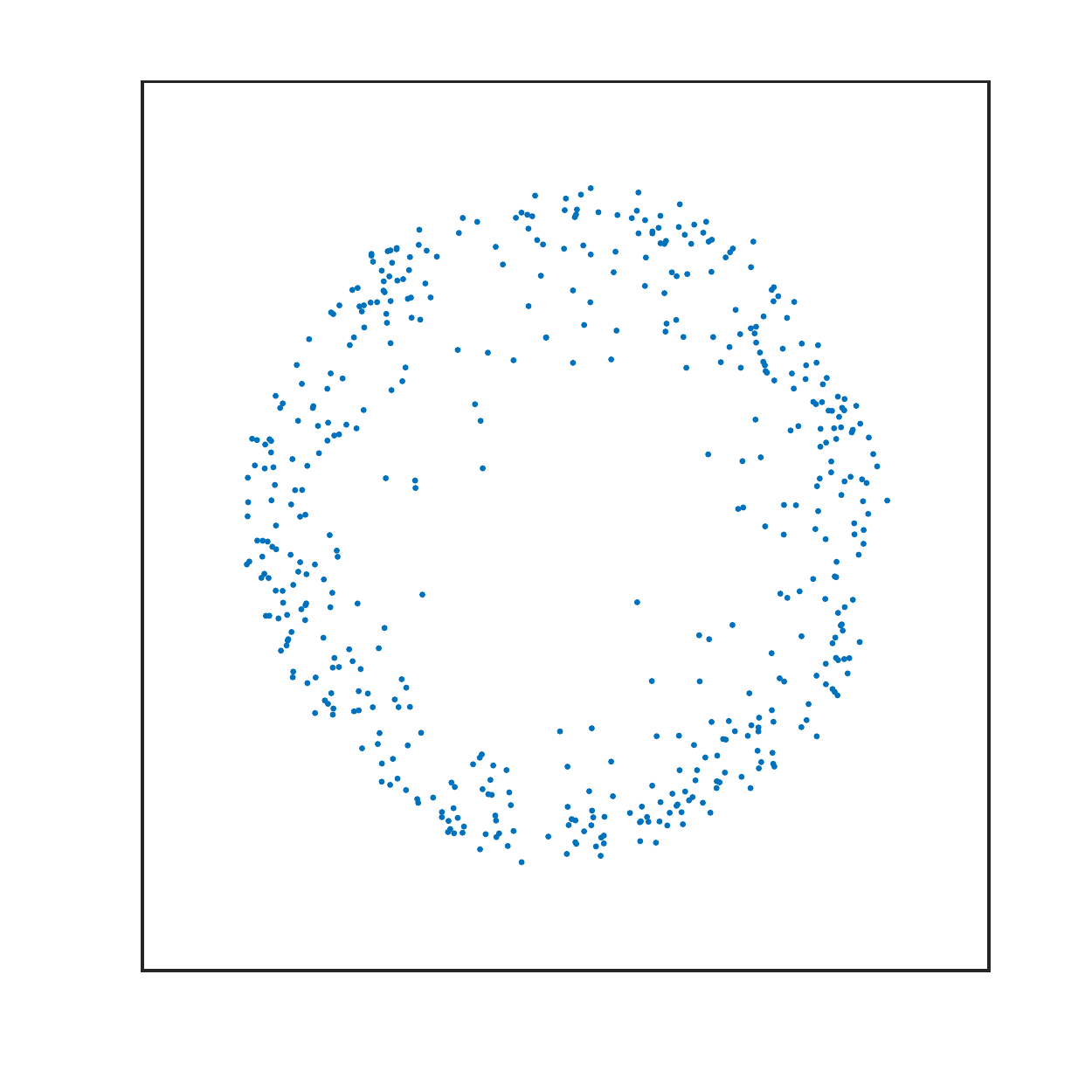}
        \caption{$\hat{y}$ from $\hat{D}$}
    \end{subfigure}	
    \begin{subfigure}[b]{0.24\textwidth}
        \includegraphics[width=\textwidth]{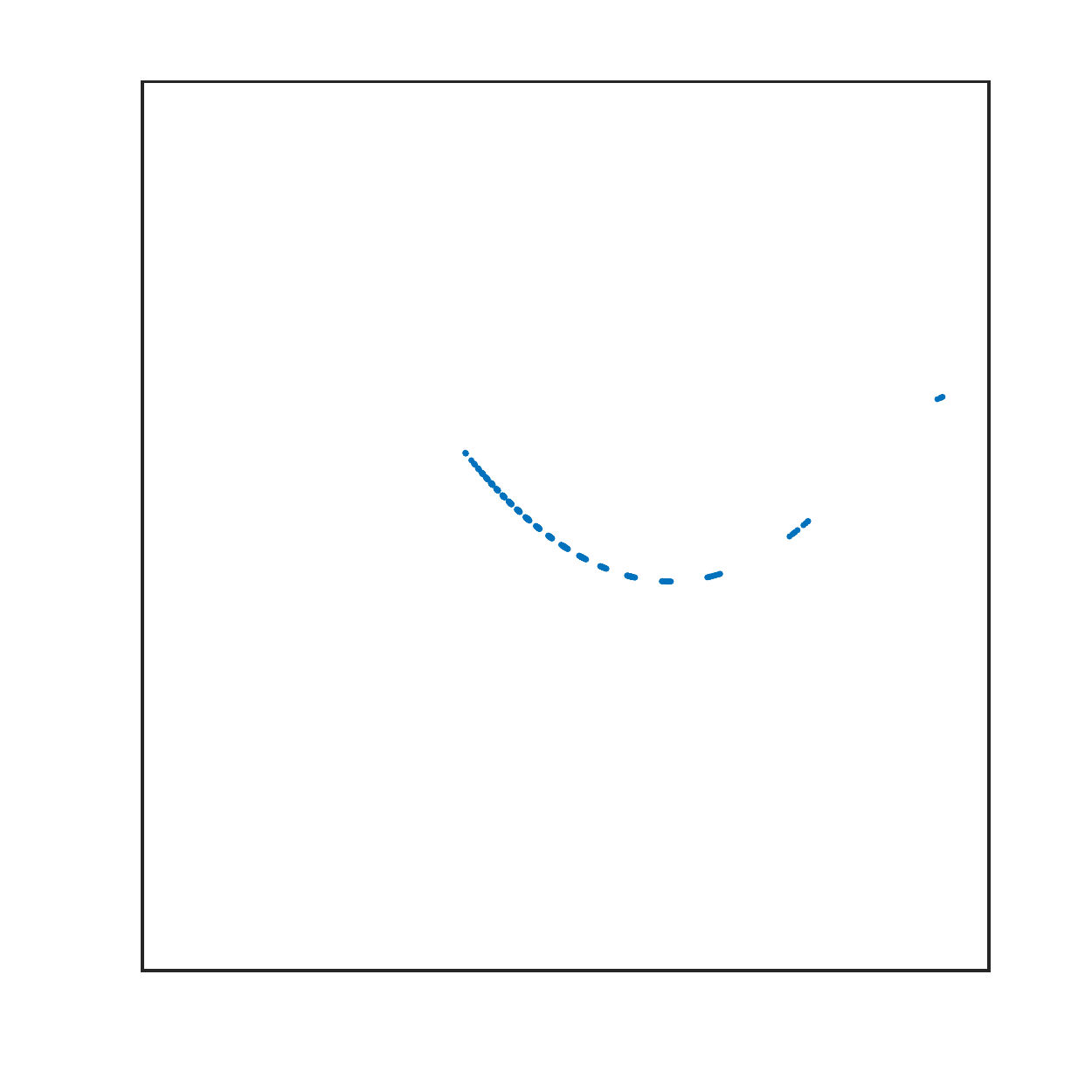}
        \caption{$\hat{y}$, USVT}
    \end{subfigure}
    \caption{Latent space model \eqref{eq: hoff}, with $n=500$, $n\rho = 20$, and latent coordinates $y_1,\ldots,y_n \in \mathbb{R}^2$ arranged in a circle. (a) Latent coordinates $y$. (b) estimated $\hat{y}$ using randomized $\tilde{D}$ from \eqref{eq: Dtilde}. (c) estimated $\hat{y}$ using MAP-style $\hat{D}$ from \eqref{eq: Dhat}. (d) estimated $\hat{y}$, using USVT spectral method \cite{chatterjee2015matrix} directly on  $A$. RMS errors for (b), (c), and (d) were $0.08, 0.07$, and $0.25$ respectively.} \label{fig: hoff}
\end{figure}

\begin{figure}
    \centering
    \begin{subfigure}[b]{0.24\textwidth}
        \includegraphics[width=\textwidth]{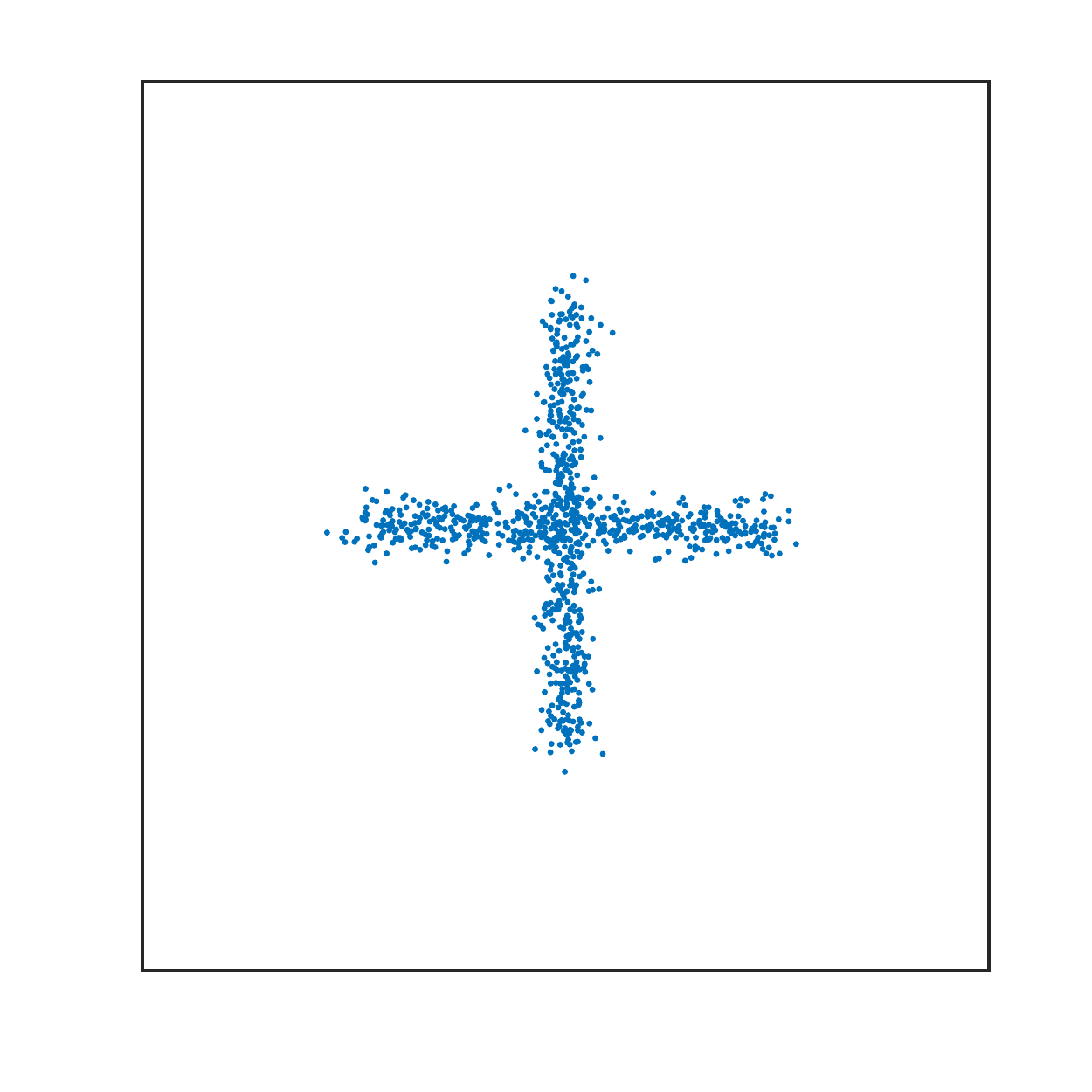}
        \caption{$y$ }
    \end{subfigure}
    \begin{subfigure}[b]{0.24\textwidth}
        \includegraphics[width=\textwidth]{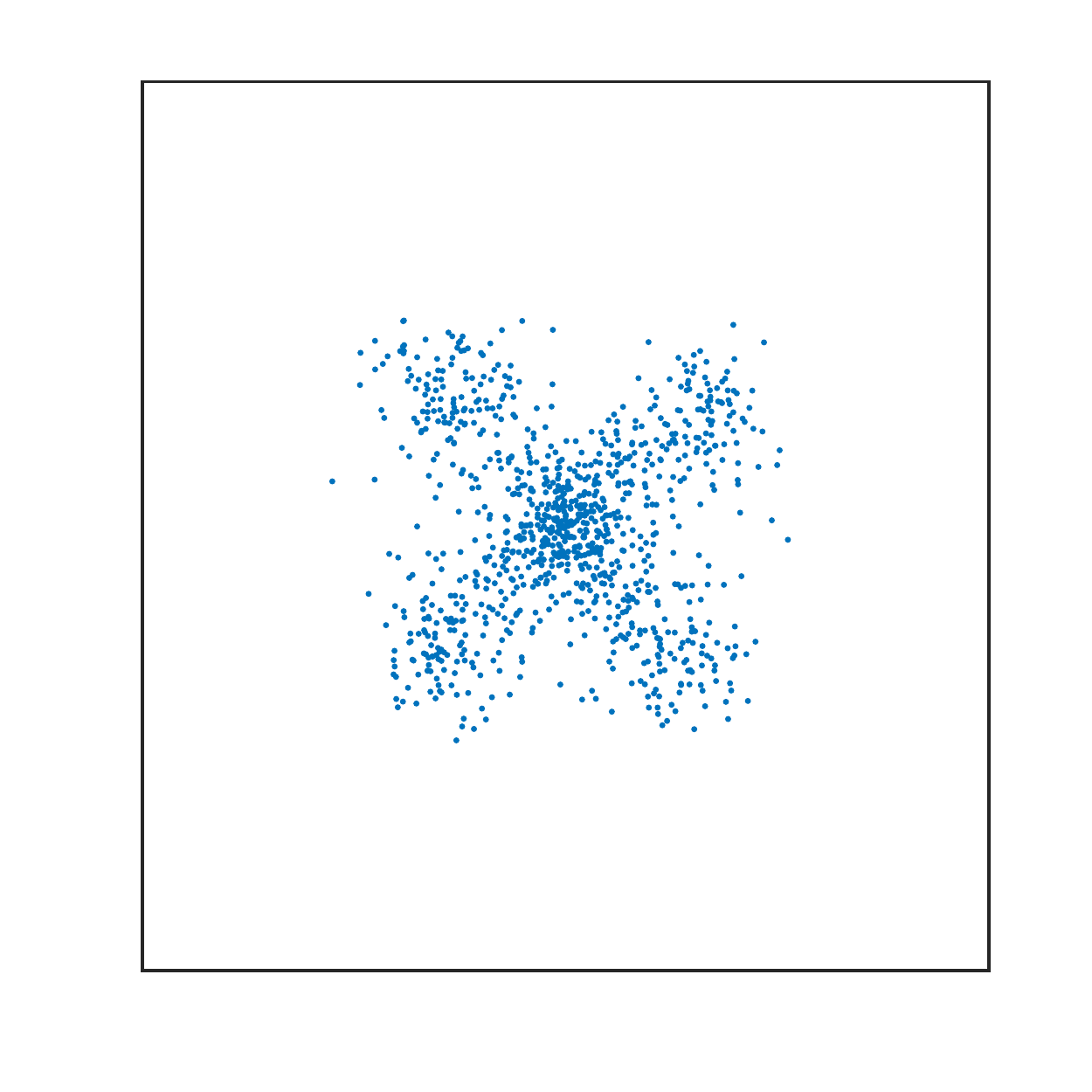}
        \caption{$\hat{y}$ from $\tilde{D}$}
    \end{subfigure}	
    \begin{subfigure}[b]{0.24\textwidth}
        \includegraphics[width=\textwidth]{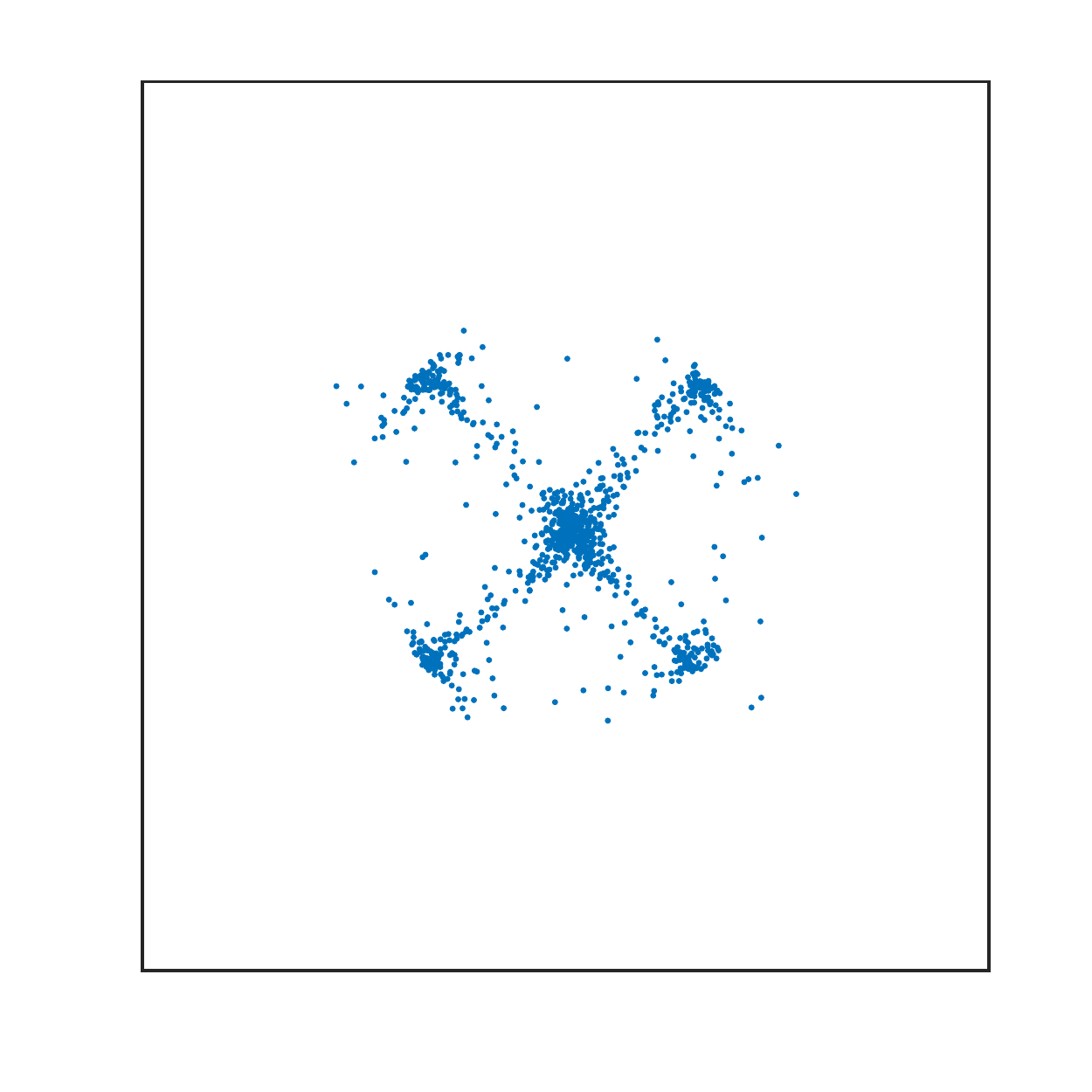}
        \caption{$\hat{y}$ from $\hat{D}$}
    \end{subfigure}	
    \begin{subfigure}[b]{0.24\textwidth}
        \includegraphics[width=\textwidth]{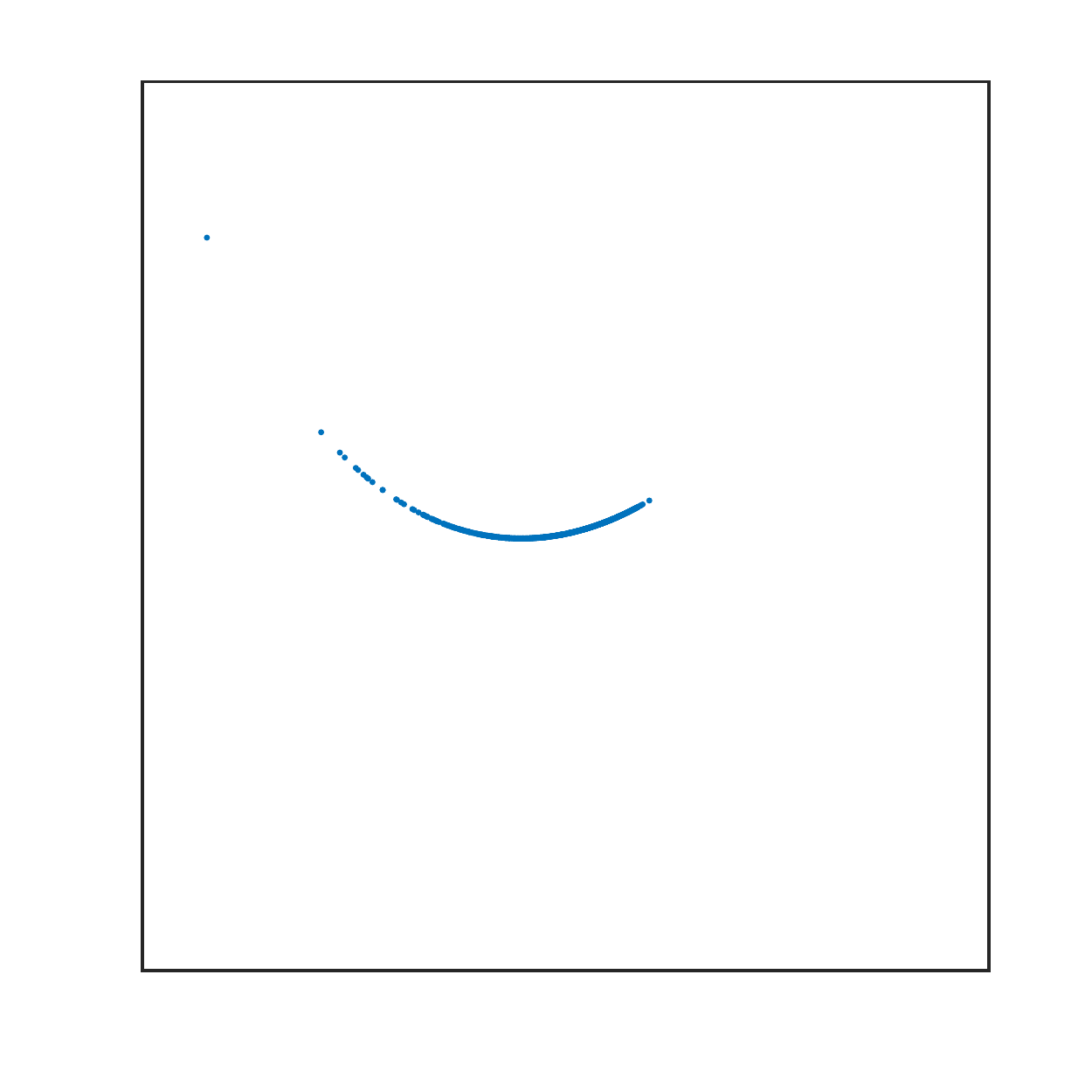}
        \caption{$\hat{y}$, USVT}
    \end{subfigure}
    \caption{Latent space model \eqref{eq: hoff}, with $n=1000$, $n\rho = 140$, and latent coordinates $y_1,\ldots,y_n \in \mathbb{R}^2$ arranged in a cross. (a) Latent coordinates $y$. (b) estimated $\hat{y}$ using randomized $\tilde{D}$ from \eqref{eq: Dtilde}. (c) estimated $\hat{y}$ using MAP-style $\hat{D}$ from \eqref{eq: Dhat}. (d) estimated $\hat{y}$, using USVT spectral method \cite{chatterjee2015matrix} directly on  $A$. RMS errors for (b), (c), and (d) were $0.044, 0.040$, and $0.11$ respectively.} \label{fig: hoff2}
\end{figure}

\begin{figure}
    \centering
    \begin{subfigure}[b]{0.24\textwidth}
        \includegraphics[width=\textwidth]{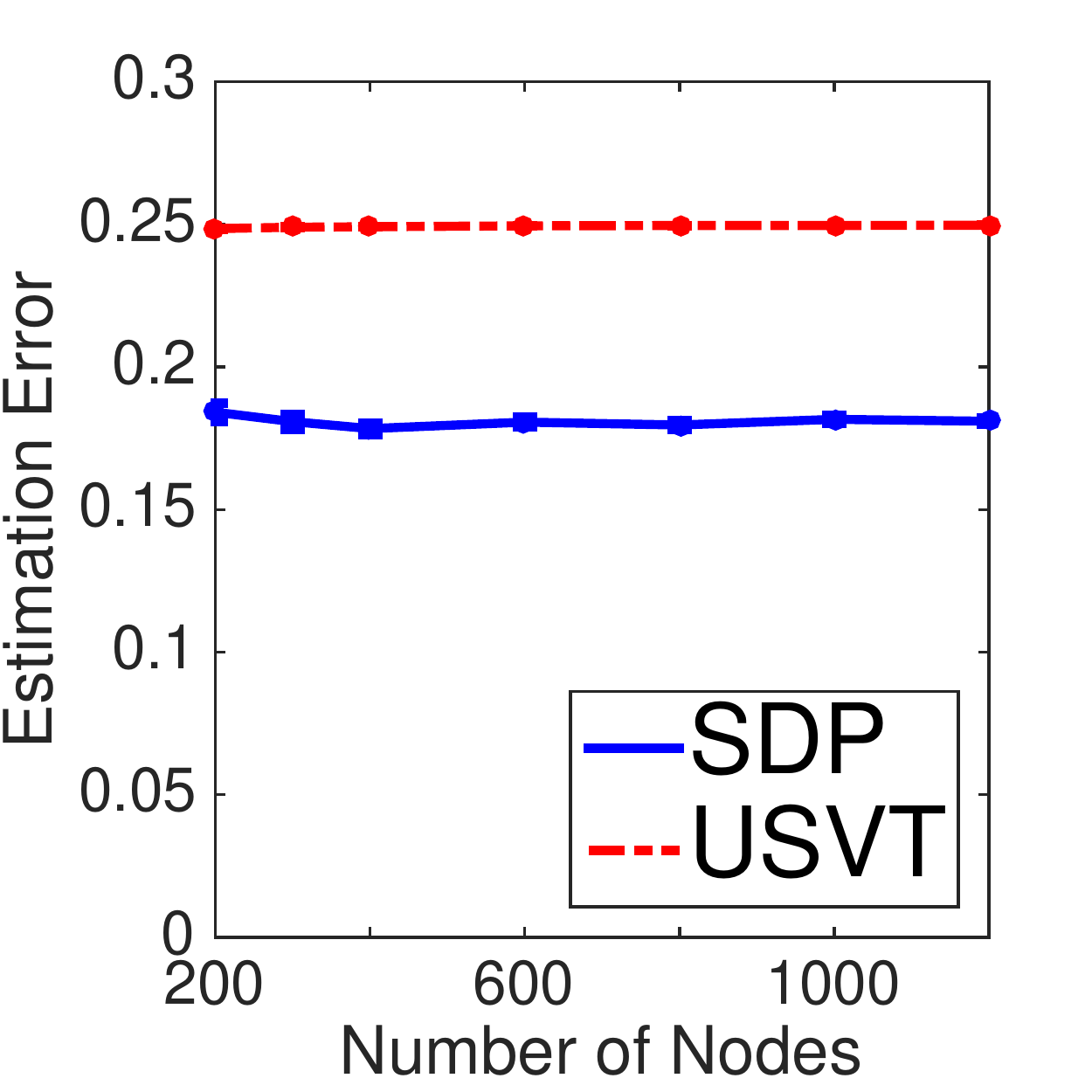}
        \caption{$n\rho = 5$ }
    \end{subfigure}
    \begin{subfigure}[b]{0.24\textwidth}
        \includegraphics[width=\textwidth]{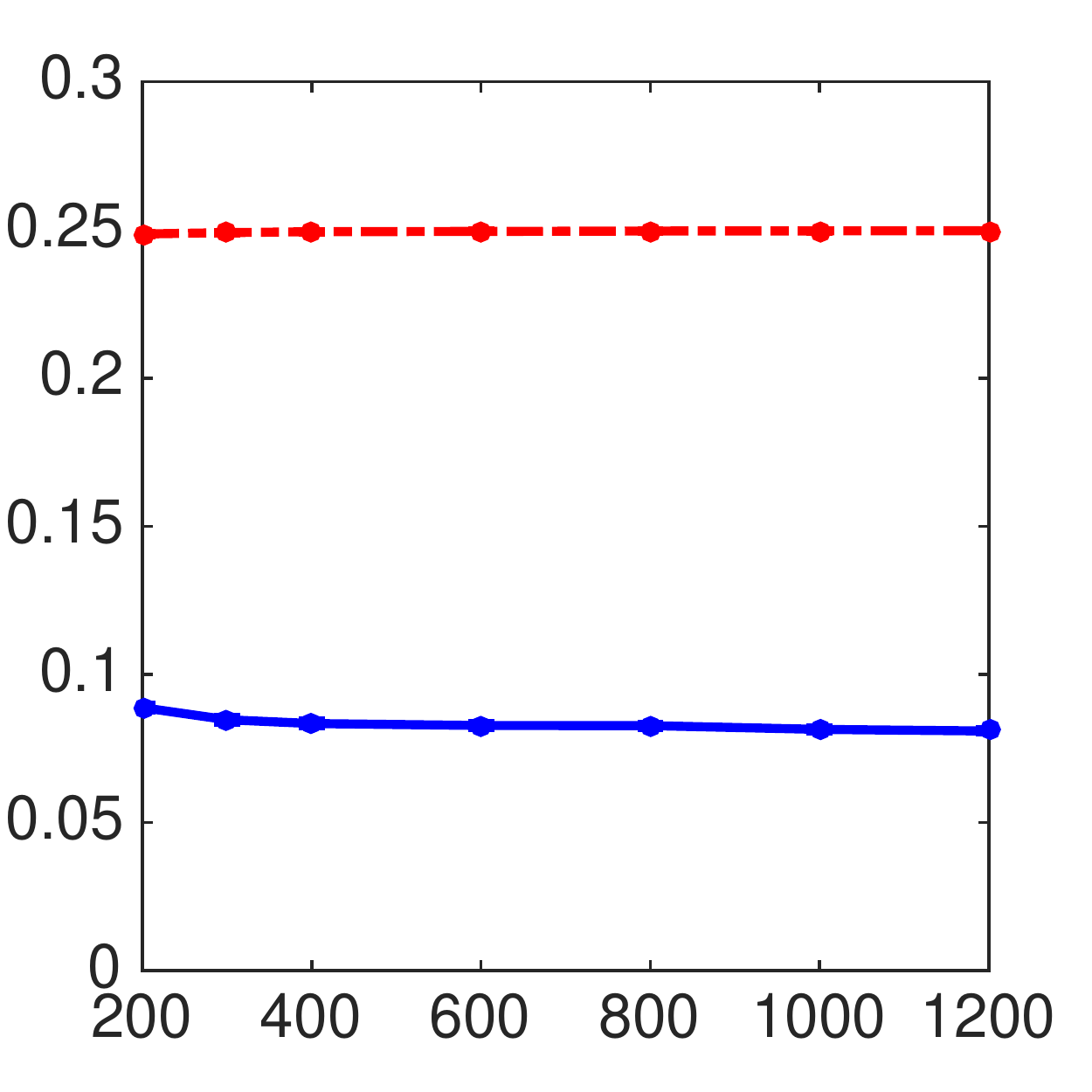}
        \caption{$n\rho = 15$ }
    \end{subfigure}
    \begin{subfigure}[b]{0.24\textwidth}
        \includegraphics[width=\textwidth]{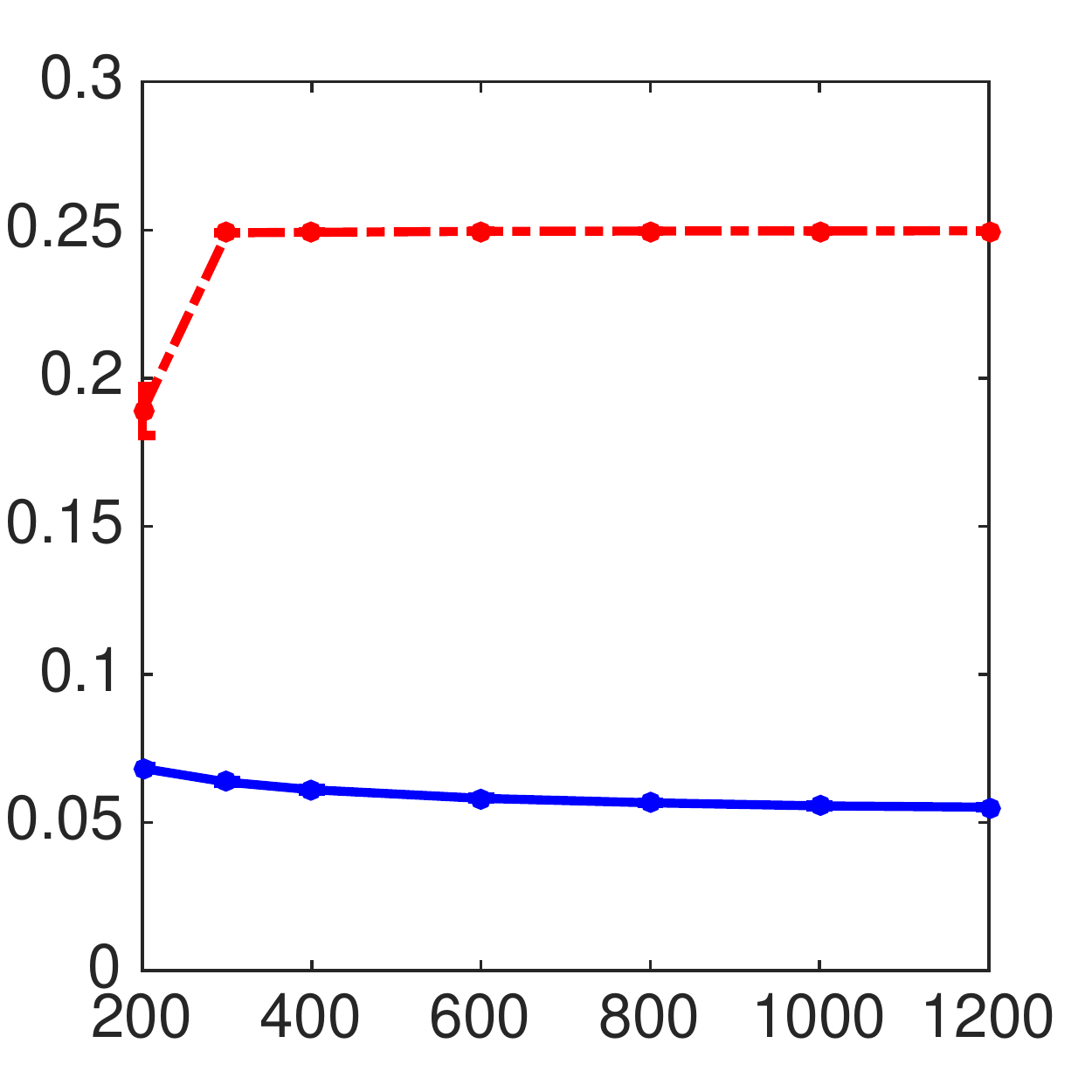}
        \caption{$n\rho = 25$ }
    \end{subfigure}
    \begin{subfigure}[b]{0.24\textwidth}
        \includegraphics[width=\textwidth]{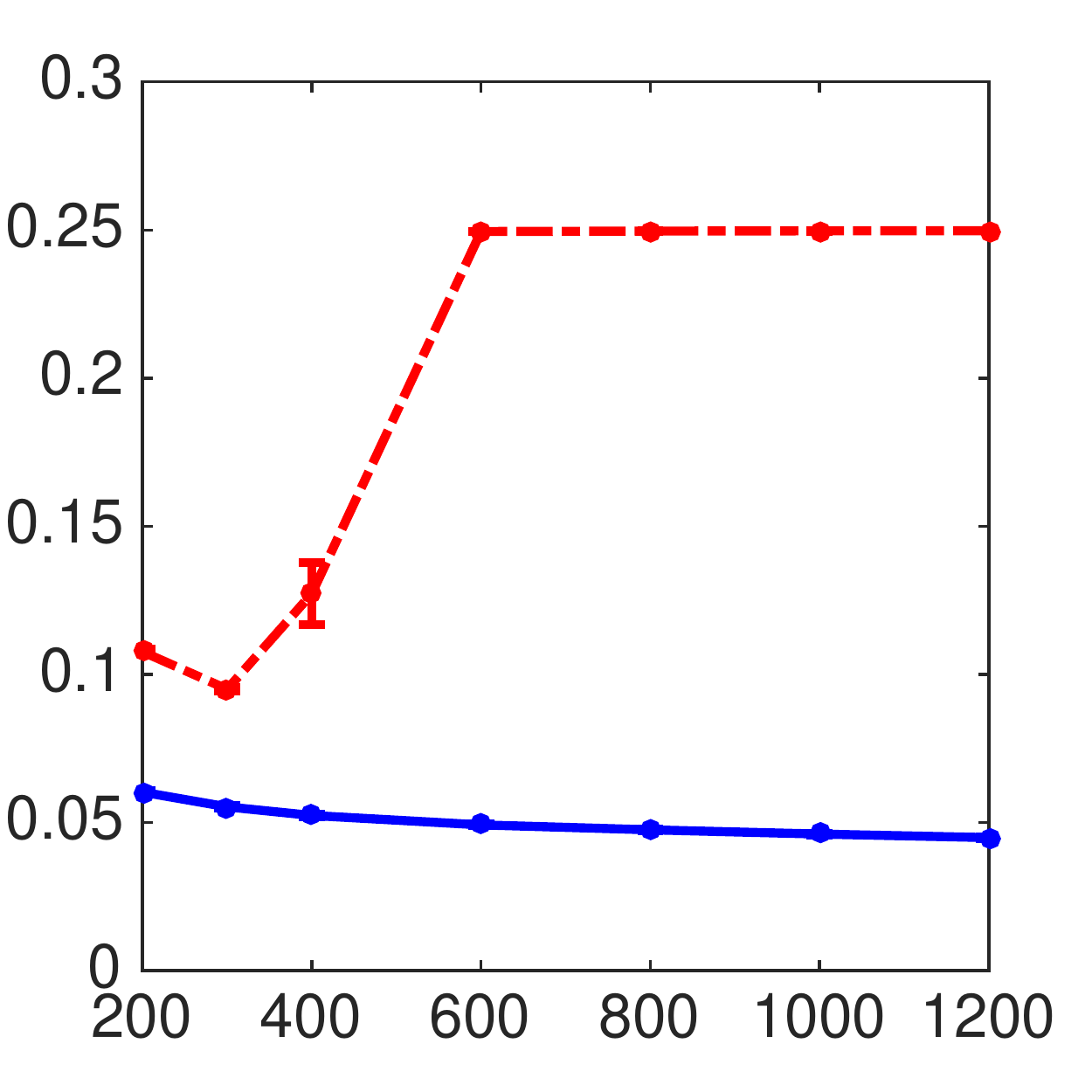}
        \caption{$n\rho = 35$ }
    \end{subfigure}
    \caption{Estimation error for $\hat{y}$ found by semidefinite programing (blue line, solid) vs. USVT method on $A$ (red line, dotted), under the latent space model \eqref{eq: hoff} with $n = \{200, 300, 400, 600, 800, 1000, 1200\}$, expected average degree $n\rho = \{5, 15, 25, 35\}$, and latent coordinates $y_1,\ldots,y_n \in \mathbb{R}^2$ arranged in a circle. 100 simulations per trial, standard errors shown. } \label{fig: batch hoff}
\end{figure}

\appendix

\section{Proof of Theorems \ref{th: Ptilde} and \ref{th: sbm}} \label{sec: appendix A}
%\section{Proof of Theorems \ref{th: Ptilde}, \ref{th: sbm}, and Corollary \ref{cor: davis kahan}} \label{sec: appendix A}

%Section \ref{sec: appendix 1} states helper lemmas that will be used to prove the main results. Section \ref{sec: appendix 2} proves consistency results for the semidefinite program \eqref{eq: sdp}, along with proofs of the required helper lemmas. Section \ref{sec: appendix 3} proves results pertaining to numerical computation when $\hat{\theta}$ belongs to an association scheme.

\subsection{Intermediate Results} \label{sec: appendix 1}

We first present intermediate results that will be used in the proof of Theorems \ref{th: Ptilde} and \ref{th: sbm}. 
%We first present intermediate results that will be used in the proof of Theorems \ref{th: Ptilde} and \ref{th: sbm}, and Corollary \ref{cor: davis kahan}. 
Let $X^*$ denote any solution to the idealized problem
\[ \textrm{maximize } \ip{\bar{F},X} \textrm{ over all } X \in \mathcal{X}.\]
%Let $\mathcal{X}_G$ denote the set
%\[\mathcal{X}_G = \{X \in \mathbb{R}^{nK \times nK}: X \succeq 0, \operatorname{diag}(X) = 1\}.\]

Lemma \ref{le: main} is the main technical result, and states that for the general matrix model \eqref{eq: bernoulli}, $\hat{X}$ nearly optimizes the desired objective function $\bar{F}$, even though only the noisy proxy $F$ is available. Its proof closely follows the approach of \cite{guedon2015community}.

\begin{lemma} \label{le: main} 
Let Assumptions \ref{as: rho} and \ref{as: thetahat} hold. Then for some $C_1,C_2 > 0$ it holds for all $n$ that
\[ \mathbb{P}\left\{\frac{1}{n^2\rho} \ip{\bar{F},X^*-\hat{X}} \geq \frac{C_1}{\sqrt{n \rho}}\right\} = e^{-C_2n},\]
where $C_1$ and $C_2$ depend only on $K$ and $c$.
\end{lemma}

Lemma \ref{le: recovery} gives a condition under which $\hat{P}$ and $\tilde{P}$ will both be approximately block structured. It roughly states that if there exists $z \in [K]^n$ such that $(z_i,z_j) \in \arg \max_{ab} \bar{F}^{(ij)}_{ab}$ for all $i,j \in [n]$, then both $\hat{P}$ and $\tilde{P}$ will asymptotically have block structure corresponding to $z$. 
%This suggests that the semidefinite program \eqref{eq: sdp} succeeds when the hardness of the original combinatorial problem \eqref{eq: combinatoric} is solely due to the randomness of the functions $f_{ij}$.

%It requires that for each $(i,j) \in [n]^2$, the correct subset $\mathcal{Q}_{\hat{\theta}}(z_i,z_j) \subset [K]^2$ can be trivially computed if the idealized submatrix $\bar{F}^{(ij)}$ is known, by finding the largest element of $\bar{F}^{(ij)}$. Roughly speaking, the intuition is that the semidefinite program succeeds whenever the hardness of finding $z$ is solely due to the fact that $F$ is noisy.  %To be more precise, it considers settings where the correct $z$ could be found trivially from the idealized submatrices $\bar{F}^{(ij)}$ without computation.
\begin{lemma}\label{le: recovery}
Let Assumptions \ref{as: rho} and \ref{as: thetahat} hold. If there exists $z \in [K]^n$ and $\Delta > 0$ such that $P, \hat{\theta}$ and $\bar{F}$ satisfy
\begin{equation} \label{eq: label recovery}
\bar{F}_{z_i,z_j}^{(ij)} \geq \bar{F}_{ab}^{(ij)} + \rho\Delta \qquad \forall i,j \in [n] \textrm{ and } a,b \notin \mathcal{Q}_{\hat{\theta}}(z_i,z_j),
\end{equation}
	then it holds that 
	\begin{equation}\label{eq: label recovery2}
	\frac{1}{n(n-1)} \sum_{i=1}^n \sum_{j=1}^n \sum_{a=1}^K \sum_{b=1}^K \hat{X}_{ab}^{(ij)}\cdot 1\{(a,b) \notin \mathcal{Q}_{\hat{\theta}}(z_i,z_j)\} \leq O_P\left(\frac{1}{\sqrt{n\rho}}\right),
	\end{equation}
%	where $\epsilon$ satisfies $\mathbb{P}\left(\epsilon \leq \frac{C_1/\Delta}{\sqrt{n\rho}}\right) \leq e^{-C_2n}$, for $C_1$ and $C_2$ given by Theorem \ref{le: main}.
\end{lemma}

Lemma \ref{le: KL bernstein} states that the error between $P$ and the randomized estimate $\tilde{P}$ converges to its expectation.

\begin{lemma}\label{le: KL bernstein}
	Let Assumptions \ref{as: rho} and \ref{as: thetahat} hold. Then
	\begin{align*}
		\frac{1}{n(n-1)\rho} \sum_{i=1}^n \sum_{j=1}^n KL\left(P_{ij}, \tilde{P}_{ij}\right) & = \frac{1}{n(n-1)\rho} \sum_{i=1}^n \sum_{j=1}^n \mathbb{E}\left[KL\left(P_{ij}, \tilde{P}_{ij}\right)\right] \\
		& \hskip.5cm {} + O_P\left(\frac{1}{\sqrt{n\rho}}\right).
	\end{align*}
		
%	where the constants in the $O_P$ term depend only on $c$ appearing in Assumption \ref{as: thetahat}.
\end{lemma}

Bernstein's inequality states that for independent $x_1,\ldots,x_n$ satisfying $|x_i| \leq b$, with variance $\sigma_i^2$ and expectation $\mathbb{E}x_i = 0$, it holds that
\[\mathbb{P}\left(\frac{1}{n}\sum_{i=1}^n x_i \geq t\right) \leq \exp\left(-\frac{nt^2}{\frac{2}{n}\sum_{i=1}^n \sigma_i^2 + bt/3}\right).\]

Grothendieck's inequality \cite[Th. 3.1]{guedon2015community} states that there exists a universal constant $C_G$ such that for any matrix $M \in \mathbb{R}^{n \times n}$,
\begin{equation} \label{eq: grothendeick}
\max_{X \in \mathcal{M}} | \ip{M, X} | \leq C_G \max_{s,t \in \{-1,1\}^n} | s^T M t|,
\end{equation}
where $\mathcal{M} = \{UV^T: \textrm{ all rows of } U, V \in B_2^n\}$, and $B_2^n = \{x \in \mathbb{R}^n: \|x\|_2 \leq 1\}$ is the $n$-dimensional unit ball. 

Lemmas \ref{le: main}, \ref{le: recovery}, and \ref{le: KL bernstein} are proven in Section \ref{sec: lemma proofs}. 

%Lemmas \ref{le: main}, \ref{le: recovery}, \ref{le: KL bernstein}, and \ref{le: eigenval} are proven in Section \ref{sec: lemma proofs}. 

% Lemma \ref{le: helper} states that ...
% \begin{lemma} \label{le: helper}
% For all $M \in \mathbb{R}^{nK \times nK}$, it holds that
% \[ \max_{s,t \in [-1,1]^{nK}} s^TM t = \max_{s,t \in \{-1,1\}^{nK}} s^T M t.\]
% If $M$ satisfies $M_{ii} \geq 0$ for all $i \in [nK]$, it holds that
% \[ \max_{X \in \mathcal{X}} \ip{M, X} \leq \max_{X \in \mathcal{X}_G} \ip{M,X},\]
% where $\mathcal{X}$ is given by \eqref{eq: Xscr}.
% \end{lemma}

\subsection{Proof of Theorems \ref{th: Ptilde} and \ref{th: sbm}}

\begin{proof}[Proof of Theorem \ref{th: Ptilde}]
Given $z \in [K]^n$, let the vector $x(z) \in \{0,1\}^{nK}$ be given by
 \[x(z) = \left[ e_{z_1}^T \, \cdots \, e_{z_n}^T \right]^T.\]  
Theorem \ref{th: Ptilde} holds by the the following steps:	
\begin{align}
\nonumber	& \frac{1}{n(n-1)\rho} \sum_{i=1}^n \sum_{j=1}^n KL\left(P_{ij}, \tilde{P}_{ij}\right) \\
	% & \hskip1cm {} = \frac{1}{n(n-1)\rho} \sum_{i=1}^n \sum_{j=1}^n \left[ H(P_{j}) - y_{ij} - P_{ij}\log \rho\right] \label{eq: cor general 1}\\
	% & \hskip1cm {} = \frac{1}{n(n-1)\rho} \sum_{i=1}^n \sum_{j=1}^n \left[ H(P_{j}) - \mathbb{E}y_{ij} - P_{ij}\log \rho\right] + O_P\left(n^{-1/2}\right) \label{eq: cor general 2}\\
	& \hskip1cm {} = \frac{1}{n(n-1)\rho} \sum_{i=1}^n \sum_{j=1}^n \mathbb{E}\left[KL\left(P_{ij}, \tilde{P}_{ij}\right)\right] + O_P\left(\frac{1}{\sqrt{n\rho}}\right) \label{eq: cor general 3}\\
	& \hskip1cm {} = \frac{1}{n(n-1)\rho} \sum_{i=1}^n \sum_{j=1}^n H(P_{ij}) - \ip{\bar{F}, \hat{X}} + O_P\left(\frac{1}{\sqrt{n\rho}}\right) \label{eq: cor general 4}\\
	& \hskip1cm {} \leq \frac{1}{n(n-1)\rho} \sum_{i=1}^n \sum_{j=1}^n H(P_{ij}) - \ip{\bar{F}, X^*} + O_P\left(\frac{1}{\sqrt{n\rho}}\right) \label{eq: cor general 5} \\
	& \hskip1cm {} \leq \frac{1}{n(n-1)\rho} \sum_{i=1}^n \sum_{j=1}^n H(P_{ij}) - \max_{z\in [K]^n} \ip{\bar{F}, x(z)x(z)^T} + O_P\left(\frac{1}{\sqrt{n\rho}}\right) \label{eq: cor general 6}\\
 	& \hskip1cm {} \leq \min_{z \in [K]^n}  \frac{1}{n(n-1)\rho} \sum_{i=1}^n \sum_{j=1}^n KL\left(P_{ij}, \hat{\theta}_{z_i,z_j}\right) + O_P\left(\frac{1}{\sqrt{n\rho}}\right)  \label{eq: cor general 7}
\end{align}
where \eqref{eq: cor general 3} holds by Lemma \ref{le: KL bernstein}; \eqref{eq: cor general 4} and \eqref{eq: cor general 7} follow from the identity
\[ KL(P_{ij}, \hat{\theta}_{ab}) = H(P_{ij}) - \bar{F}_{ab}^{(ij)},\]
%\[	\sum_{i=1}^n \sum_{j=1}^n \mathbb{E}\left[KL\left(P_{ij}, \hat{P}_{ij}\right)\right] = \sum_{i=1}^n \sum_{j=1}^n H(P_{ij}) - \ip{\bar{F},\hat{X}},\]
with \eqref{eq: cor general 4} additionally using \eqref{eq: Phat randomized}, the definition of $\tilde{P}$; \eqref{eq: cor general 5} holds by Lemma \ref{le: main}; and \eqref{eq: cor general 6} holds because $x(z)x(z)^T \in \mathcal{X}$, implying that $\ip{\bar{F}, X^* - x(z)x(z)^T} \geq 0$.

\end{proof}

\begin{proof}[Proof of Theorem \ref{th: sbm}]
	Define $\hat{\pi}$ by
	\[\hat{\pi}_a = \frac{1}{n} \sum_{i=1}^n 1\{z_i=a\}.\]
	By Bernstein's inequality, it can be seen that $\hat{\alpha} = \rho(1+o_P(1))$, and that $\rho = \frac{1}{{n \choose 2}} \sum_{i<j} P_{ij}$ satisfies
	\begin{align*}
		\rho & = \sum_{a=1}^K \sum_{b=1}^K \hat{\pi}_a \hat{\pi}_b \theta^*_{ab}(1 + o(1)) \\
		& = \alpha(1 + o_P(1)),
	\end{align*}
	so that $\hat{\alpha} \rightarrow \rho \rightarrow \alpha$. As a result, Assumptions \ref{as: rho} and \ref{as: thetahat} can be seen to hold (in probability). Let $\Delta'$ be defined by
	\[ \Delta' = \min_{a,b \in [K]^2} \min_{c,d \notin \mathcal{Q}_{B^*}} B_{ab}^* \log \frac{\hat{B}_{ab}}{\hat{B}_{cd}} - (\hat{B}_{ab} - \hat{B}_{cd}).\]
	It can be seen that $\Delta' > 0$ by assumption. To bound $\bar{F}_{z_i,z_j}^{(ij)} - \bar{F}_{cd}^{(ij)}$ so as to apply Lemma \ref{le: recovery}, observe for all $i,j \in [n]$ and $c,d \notin \mathcal{Q}_{\hat{\theta}}(z_i,z_j)$: 
	\begin{align*}
		\bar{F}_{z_i,z_j}^{(ij)} - \bar{F}_{cd}^{(ij)} & = P_{ij} \log \frac{\hat{\theta}_{z_i,z_j}}{\hat{\theta}_{cd}} + (1- P_{ij}) \log \frac{1 - \hat{\theta}_{z_i,z_j}}{1 - \hat{\theta}_{cd}} \\
		& = \alpha B^*_{z_i,z_j} \log \frac{\hat{B}_{z_i,z_j}}{\hat{B}_{cd}} + \log \frac{1 - \hat{\alpha}\hat{B}_{z_i,z_j}}{1 - \hat{\alpha}\hat{B}_{cd}} - \alpha B^*_{z_i,z_j} \log \frac{1 - \hat{\alpha}\hat{B}_{z_i,z_j}}{1 - \hat{\alpha}\hat{B}_{cd}} \\
		&= \alpha B^*_{z_i,z_j} \log \frac{\hat{B}_{z_i,z_j}}{\hat{B}_{cd}} + \hat{\alpha}(\hat{B}_{z_i,z_j} - \hat{B}_{cd}) + O(\hat{\alpha}^2) + O(\alpha \hat{\alpha}) \\
		& = \rho B^*_{z_i,z_j} \log \frac{\hat{B}_{z_i,z_j}}{\hat{B}_{cd}} - \rho(\hat{B}_{z_i,z_j} - \hat{B}_{cd}) + o_P(\alpha) \\
		& \geq \rho \Delta' + o_P(\alpha),
	\end{align*}
	where the $o_P(\alpha)$ terms are bounded uniformly over all $(i,j) \in [n]^2$ and $(c,d) \in [K]^4$. This implies that for all $i,j \in [n]$ and all $c,d \notin \mathcal{Q}_{B^*}(z_i,z_j)$, it holds that
	\[\bar{F}_{z_i,z_j}^{(ij)} - \bar{F}_{cd}^{(ij)} \geq \rho(\Delta' + o_P(1)),\]
	implying that the conditions of Lemma \ref{le: recovery} will hold in probability for any $\Delta < \Delta'$. Lemma \ref{le: recovery} thus implies
	\begin{equation}
	\frac{1}{n^2} \sum_{i=1}^n \sum_{j=1}^n \sum_{a=1}^K \sum_{b=1}^K \hat{X}^{(ij)}_{ab}\cdot 1\{(a,b) \notin \mathcal{Q}_{B^*}(z_i,z_j)\}  = O_P\left(\frac{1}{\sqrt{n\alpha}}\right). \label{eq: sbm bound}
	\end{equation}
	Let $\mathcal{E}_{ij}$ for $i,j \in [n]$ be given by
	\[ \mathcal{E}_{ij} = \sum_{a=1}^K \sum_{b=1}^K \hat{X}_{ab}^{(ij)}\cdot 1\{(a,b) \notin \mathcal{Q}_{B^*}(z_i,z_j)\}. \]

	To show \eqref{eq: th sbm 1}, we apply \eqref{eq: sbm bound} as follows: 
	\begin{align*}
	\frac{1}{n^2}\sum_{i=1}^n \sum_{j=1}^n 1\{\hat{P}_{ij} \neq \hat{\theta}_{z_iz_j}\} & \leq \frac{1}{n^2} \sum_{i=1}^n \sum_{j=1}^n 1\{ \mathcal{E}_{ij} \geq 1/2\} \\
	& \leq \frac{1}{n^2} \sum_{i=1}^n \sum_{j=1}^n 2\mathcal{E}_{ij} \\
	& = O_P\left(\frac{1}{\sqrt{n\alpha}}\right),
	\end{align*}
	where the last equality follows from \eqref{eq: sbm bound}. 
	
	To show that \eqref{eq: th sbm 1} holds with $\tilde{P}$ in place of $\hat{P}$, we observe the following bounds:
	\begin{align*}
	 	\mathbb{E}\left[\frac{1}{n^2} \sum_{i=1}^n \sum_{j=1}^n 1\{\tilde{P}_{ij} \neq \hat{\theta}_{z_iz_j}\}\right] &= \frac{1}{n^2} \sum_{i=1}^n \sum_{j=1}^n \mathcal{E}_{ij} \\
		& = O_P\left(\frac{1}{\sqrt{n\alpha}}\right)
	\end{align*}
	and
	\begin{align*}
		\operatorname{Var}\left[\frac{1}{n^2} \sum_{i=1}^n \sum_{j=1}^n 1\{\tilde{P}_{ij} \neq \hat{\theta}_{z_iz_j} \}\right] & 
		= \frac{4}{n^4} \sum_{i < j} \operatorname{Var}\left[1\{\tilde{P}_{ij} \neq \theta_{ij}\}\right] \\
		& \leq \frac{4}{n^4} \sum_{i < j} \mathbb{P}\left( \tilde{P}_{ij} \neq \theta_{ij}\right) \\
		& = \frac{4}{n^4} \sum_{i <j} \mathcal{E}_{ij}  \\
		& = \frac{1}{n^2} O_P\left(\frac{1}{\sqrt{n\alpha}}\right).
	\end{align*}
Applying Chebychev, which states that $\mathbb{P}\left(|X - \mathbb{E}X| \geq k \sqrt{\operatorname{Var}X}\right) \leq k^{-2}$, with $k = n^{3/4}\alpha^{-1/4}$ yields that 
	\[\frac{1}{n^2} \sum_{i=1}^n \sum_{j=1}^n 1\{\tilde{P}_{ij} \neq \hat{\theta}_{z_iz_j}\} = O_P\left(\frac{1}{\sqrt{n\alpha}}\right).\]
\end{proof}

\subsection{Proof of Lemmas \ref{le: main}, \ref{le: recovery}, and \ref{le: KL bernstein}} \label{sec: lemma proofs}

\begin{proof}[Proof of Lemma \ref{le: main}]

Let $\tilde{F} \in \mathbb{R}^{nK \times nK}$ be a re-centered version of $\bar{F}$, with submatrices $\tilde{F}^{(ij)} \in \mathbb{R}^{K \times K}$ given by
\[ \tilde{F}^{(ij)} = \bar{F}^{(ij)} + (A_{ij} - P_{ij})\log \rho \cdot \obf \obf^T.\]
%\[ \tilde{F}^{(ij)}_{ab} = P_{ij} \log \hat{\theta}_{ab} + (1-P_{ij}) \log(1-\hat{\theta}_{ab}) + (A_{ij} - P_{ij})\log \rho.\]
Alegebraic manipulation yields that
\begin{align}
\nonumber \ip{\bar{F},X^*} - \ip{\bar{F},\hat{X}} & = \ip{\tilde{F},X^*} + \ip{\bar{F} -\tilde{F},X^*} - \ip{\tilde{F},\hat{X}} - \ip{\bar{F} - \tilde{F},\hat{X}} \\
 & =  \ip{\tilde{F},X^*} - \ip{\tilde{F},\hat{X}} \label{eq: th1 recenter} \\
%\nonumber & =\ip{\bar{F},X^*} - \ip{F,X^*} + \ip{F,X^*} - \ip{F,\hat{X}} \\
%\nonumber & \hskip.5cm {} + \ip{F,\hat{X}} -  \ip{\bar{F},\hat{X}} \\
	& \leq \ip{\tilde{F},X^*} - \ip{F,X^*} + \ip{F,\hat{X}} - \ip{\tilde{F},\hat{X}} \label{eq: th1 basic eq} \\
\nonumber	& \leq 2 \max_{X \in \mathcal{X}} \left| \ip{\bar{F} - F, X}\right| \\
%& \leq \max_{X \in \mathcal{X}_G} 2\left| \ip{F - \tilde{F},X} \right| \label{eq: th1 le1a} \\
%	& \leq C_G \max_{s,t \in [-1,1]^{nK}} s^T(F - \tilde{F})t \label{eq: th1 grothendeick} \\
	& \leq 2C_G \max_{s,t \in \{-1,1\}^{nK}} s^T(F - \tilde{F})t \label{eq: th1 le1b},
\end{align}
where \eqref{eq: th1 recenter} holds because $\tilde{F}^{(ij)} - \bar{F}^{(ij)} \propto \obf \obf^T$ for all $i,j \in [n]$, implying that $\ip{\tilde{F} - \bar{F},X} = \ip{\tilde{F} - \bar{F},X'}$ for all $X, X' \in \mathcal{X}$; \eqref{eq: th1 basic eq} holds because $\hat{X}$ maximizes $\ip{F, X}$, implying that $\ip{F, X^* - \hat{X}} \leq 0$; and \eqref{eq: th1 le1b} follows by Grothendieck's inequality and because $\mathcal{X} \subset \mathcal{M}$.
%\eqref{eq: th1 le1a} and \eqref{eq: th1 le1b} hold by Lemma \ref{le: helper}; and \eqref{eq: th1 grothendeick} follows by Grothendieck's inequality and because $\mathcal{X}_G \subset \mathcal{M}$.

It remains to bound the right hand side of \eqref{eq: th1 le1b}. By the definition of $\tilde{F}$, it can be seen that
\[F_{ab}^{(ij)} - \tilde{F}_{ab}^{(ij)} = (A_{ij} - P_{ij}) \log \frac{\hat{\theta}_{ab}/\rho}{1 - \hat{\theta}_{ab}},\]
so that 
\[s^T (F - \tilde{F})t = \sum_{i=1}^n \sum_{j=1}^n \sum_{a=1}^K \sum_{b=1}^K s_a^{(i)} t_b^{(j)} (A_{ij} - P_{ij}) \log \frac{\hat{\theta}_{ab}/\rho}{1 - \hat{\theta}_{ab}}.\]
Given $s,t \in \{-1,1\}^{nK}$ define $x_{ij} \equiv x_{ij}(s,t)$ by
\[ x_{ij} = (A_{ij} - P_{ij}) \sum_{a=1}^K \sum_{b=1}^K  \left(s_a^{(i)}t_b^{(j)} + s_a^{(j)}t_b^{(j)}\right)\log \frac{\hat{\theta}_{ab}/\rho}{1 - \hat{\theta}_{ab}},\]
so that $s^T (F - \tilde{F})t = \sum_{i<j} x_{ij}$. Using $\rho/c \leq \hat{\theta}_{ab}/(1-\hat{\theta}_{ab}) \leq c\rho$ which holds by Assumption \ref{as: thetahat} and letting $C = 2|\log c|$, it can be seen that
 \begin{align*} 
 \operatorname{Var}(x_{ij} + x_{ji}) \leq  P_{ij}(1- P_{ij}) K^4 C^2 \qquad \textrm{ and } \qquad |x_{ij}| \leq K^2 C.
 \end{align*}
 By Bernstein's inequality, it follows that
\[ \mathbb{P}\left[ \frac{1}{{n\choose 2}} \sum_{i,j:i<j}  x_{ij} \geq \epsilon\right] \leq \exp\left(-\frac{{n \choose 2}\epsilon^2}{K^4 C^2 \rho + 2K^2 C \epsilon}\right),\]
where we have used the fact that ${{n \choose 2}}^{-1} \sum_{i<j} P_{ij}(1-P_{ij}) \leq \rho$. Letting $\epsilon = C_1\sqrt{\frac{\rho}{n}}$ for any value of $C_1$ implies 
\[ \mathbb{P}\left[ \frac{1}{{n\choose 2}} \sum_{i,j:i<j}  x_{ij} \geq C_1\sqrt{\frac{\rho}{n}} \right] \leq \exp\left(-\frac{\frac{1}{2}nC_1^2}{K^4 C^2 + 2K^2C}\right).\]
Applying a union bound over all $s,t \in \{-1,1\}^{nK}$ implies
\[ \mathbb{P}\left[ \max_{s,t \in \{-1,1\}^{nK}} \frac{1}{{n\choose 2}} \sum_{i,j:i<j}  x_{ij} \geq C_1\sqrt{\frac{\rho}{n}}\right] \leq 2^{nK} \exp\left(-\frac{\frac{1}{2}nC_1^2}{K^4 C^2 + 2K^2C}\right),\]
which implies for all $C_1$ satisfying $\frac{1}{2} C_1^2/(K^4C^2 + 2K^2CC_1) > K \log 2$ that 
\begin{equation}\label{eq: concentration}
 \mathbb{P}\left[ \max_{s,t \in \{-1,1\}^{nK}} \frac{1}{{n\choose 2}} \sum_{i,j:i<j}  x_{ij} \geq C_1\sqrt{\frac{\rho}{n}} \right] = e^{-C_2n},
\end{equation}
where $C_2 = \frac{1}{2} C_1^2/(K^4 C^2 + 2K^2CC_1) - K\log 2$. Since $s^T(F - \tilde{F})t = \sum_{i<j} x_{ij}$, combining \eqref{eq: concentration} and \eqref{eq: th1 le1b} proves the lemma.
\end{proof}

\begin{proof}[Proof of Lemma \ref{le: recovery}]
	It holds that
	\begin{align*}
		\ip{\bar{F},\hat{X}} & = \sum_{i=1}^n \sum_{j=1}^n \sum_{a=1}^K \sum_{b=1}^K \bar{F}_{ab}^{(ij)} \hat{X}_{ab}^{(ij)} \\
		& = \sum_{i=1}^n \sum_{j=1}^n \left[\bar{F}^{(ij)}_{z_i,z_j} + \sum_{a=1}^K \sum_{b=1}^K \left(\bar{F}_{ab}^{(ij)} - \bar{F}_{z_i,z_j}^{(ij)}\right) \hat{X}_{ab}^{(ij)}\right] \\
		& \leq \sum_{i=1}^n \sum_{j=1}^n \left[ \bar{F}^{(ij)}_{z_i,z_j} - \rho \Delta \sum_{a=1}^K \sum_{b=1}^K \hat{X}_{ab}^{(ij)}\cdot 1\{(a,b) \notin \mathcal{Q}_{\hat{\theta}}(z_i,z_j)\}\right].
	\end{align*}
	Rearranging and using $\ip{\bar{F}, X^*} = \sum_{ij} \bar{F}^{(ij)}_{z_i,z_j}$ (as implied by \eqref{eq: label recovery}) yields
	\[ \frac{\ip{\bar{F},X^*} - \ip{\bar{F},\hat{X}}}{\rho\Delta} \geq \sum_{i=1}^n \sum_{j=1}^n \sum_{a=1}^K \sum_{b=1}^K \hat{X}_{ab}^{(ij)}\cdot 1\{(a,b) \notin \mathcal{Q}_{\hat{\theta}}(z_i,z_j)\},\]
	and dividing both sides by $n^2$ and using Lemma \ref{le: main} yields that with probability at least $1-e^{-C_2n}$,
	\[\frac{C_1/\Delta}{\sqrt{n\rho}} \geq \frac{1}{n^2}\sum_{i=1}^n \sum_{j=1}^n \sum_{a=1}^K \sum_{b=1}^K \hat{X}_{ab}^{(ij)}\cdot 1\{(a,b) \notin \mathcal{Q}_{\hat{\theta}}(z_i,z_j)\},\]
	proving the lemma.
\end{proof}

\begin{proof}[Proof of Lemma \ref{le: KL bernstein}]
Let $x_{ij}$ for $i,j \in [n]$ be given by
\[ x_{ij} = H(P_{ij}) - P_{ij}\log \rho - KL\left(P_{ij}, \tilde{P}_{ij}\right).\]	
It can be seen that $x_{ij}$ for $i<j$ are independent random variables with distributions
\[ x_{ij} = P_{ij} \log \frac{\hat{\theta}_{ab}/\rho}{1 - \hat{\theta}_{ab}} + \log (1-\hat{\theta}_{ab}) \qquad \textrm{ with probability } \hat{X}_{ab}^{(ij)},\]
with $x_{ji} = x_{ij}$ for $i > j$ and $x_{ii} = 0$. By definition of $\{x_{ij}\},$ it holds that
% It can be seen by inspection that
% \begin{align*}
% KL\left(P_{ij}, \tilde{P}_{ij}\right) &= H(P_{ij}) - x_{ij} - P_{ij}\log \rho,
% \end{align*}
% and hence that
\begin{align}
	\sum_{i=1}^n \sum_{j=1}^n \left( KL\left(P_{ij}, \tilde{P}_{ij}\right) -  \mathbb{E}\left[KL\left(P_{ij}, \tilde{P}_{ij}\right)\right] \right) &= \sum_{i=1}^n \sum_{j=1}^n \left( \mathbb{E}x_{ij} - x_{ij} \right). \label{eq: proof KL bernstein}
\end{align}
To bound the right hand side of \eqref{eq: proof KL bernstein}, observe that by Assumption \ref{as: thetahat}
\begin{align*}
	|x_{ij}| & \leq P_{ij} |\log c| + | \log(1+c\rho)| \\
	& \leq P_{ij} |\log c| + \frac{c\rho}{1+c\rho}
\end{align*}
and hence it holds that
\begin{align*}
	\max_{ij} |x_{ij} - \mathbb{E}x_{ij}| &\leq O(1) \\
	\frac{1}{n(n-1)} \sum_{i,j:i<j} \operatorname{Var}(x_{ij} + x_{ji}) &\leq O(\rho),
\end{align*}
Applying Bernstein's inequality thus yields 
\begin{align*}
	\mathbb{P}\left( \frac{1}{{n \choose 2}} \sum_{i,j:i<j} (x_{ij} - \mathbb{E}x_{ij}) \geq \epsilon \right) \leq \exp\left(-\frac{{n \choose 2}\epsilon^2}{O(\rho) + O(1)\epsilon}\right),
\end{align*}
and letting $\epsilon = (\rho/n)^{1/2}$ implies that $\frac{1}{{n\choose 2}\rho} \sum_{i,j} (x_{ij} - \mathbb{E}x_{ij}) = O_P\left((n\rho)^{-1/2}\right)$. Combining this bound with \eqref{eq: proof KL bernstein} proves the lemma.
\end{proof}

\section{Proof of Theorem \ref{th: association2}} \label{sec: appendix B}

Lemma \ref{le: association} can be found as E1-E4 in \cite{goemans1999semidefinite}, and as Theorem 2.6 in \cite{bailey2004association}.

\begin{proof}[Proof of Theorem \ref{th: association2}]
Let $\mathcal{B} = \operatorname{span}(B_0,\ldots,B_\ell)$, and let $\mathcal{A}$ denote the set
\[\mathcal{A} = \{X \in \mathbb{R}^{nK \times nK}: X^{(ij)} \in \mathcal{B}\quad \forall\,  i,j\},\]
and observe the following properties of $\mathcal{A}$:
\begin{enumerate}
	\item If $M \in \mathcal{A}$, then its submatrices $M^{(ij)}$ must be symmetric and have the same eigenvectors since they are in the span of an association scheme. Specifically, if $X^{t+1} + V^t \in \mathcal{A}$ then \eqref{eq: structured Y} holds.
	\item $\mathcal{A}$ is a linear space, so if $M_1,M_2 \in \mathcal{A}$, then $M_1 + M_2 \in \mathcal{A}$ as well.
	\item $0 \in \mathcal{A}$, so the initial values $X^0,W^0,Y^0,U^0,V^0$ are in $\mathcal{A}$ 
	\item Let $B = \sum_{i=0}^\ell \gamma_i B_i$ for some weights $\gamma_0,\ldots,\gamma_\ell$. Because $B_0,\ldots,B_\ell$ are binary with disjoint support, it holds that $\max(0, B) = \sum_i \max(0,\gamma_i)B_i$ and $\log(B) = \sum_i \log(\gamma_i) B_i$.	
	\item If $\hat{\theta} \in \mathcal{B}$ then $\log \hat{\theta}$ and $\log(1-\hat{\theta})$ are in $\mathcal{B}$ as well, by property 4. This implies the matrix $F$ given by \eqref{eq: F theorem} is in $\mathcal{A}$, since each submatrix $F^{(ij)}$ is a linear combination of $\log \hat{\theta}$ and $\log (1-\hat{\theta})$.
	\item If $M \in \mathcal{A}$ then $\max(0, M)$ is also in $\mathcal{A}$, by property 4.
	\item  If $M \in \mathcal{A}$, then the eigenvectors of each submatrix are orthogonal and include the vector $1$, by Lemma \ref{le: association}. As a result, the effect of the projection $\Pi_{\mathcal{X}}(M)$ is to change the eigenvalue associated with the vector $1$ in each submatrix, implying that $\Pi_{\mathcal{X}}(M) \in \mathcal{A}$.
	\item If $X^{t+1} + V^t \in \mathcal{A}$, then \eqref{eq: structured Y} and hence \eqref{eq: Pi_S} hold, which implies that $\Pi_{\mathcal{S}_+}(X^{t+1} + V^t) \in \mathcal{A}$.
\end{enumerate}
Using these properties, we show by induction that $X^t,W^t,Y^t,U^t,V^t \in \mathcal{A}$ for all $t$. As for the base case, it holds that $X^0,W^0,Y^0,U^0,V^0 \in \mathcal{A}$ by property 3. Now suppose that $X^t,W^t,Y^t, U^t,V^t \in \mathcal{A}$ for any $t$. Then by \eqref{eq: admm} it follows that $X^{t+1} \in \mathcal{A}$ by properties 2, 5, and 7; that $W^{t+1} \in \mathcal{A}$ by properties 2 and 6; that $Y^{t+1} \in \mathcal{A}$ by properties 2 and 8; and that $U^{t+1},V^{t+1} \in \mathcal{A}$ by property 2. This completes the induction argument.

Since $X^t,W^t,Y^t,U^t,V^t \in \mathcal{A}$ for all $t$, it follows by property 2 that $X^{t+1} + V^t \in \mathcal{A}$ for all $t$. By property 1, this implies that \eqref{eq: structured Y} holds for all $t$, proving the theorem.
\end{proof}

\section{Proof of Corollary \ref{cor: eigvec}} \label{sec: appendix C}

Here we prove Corollary \ref{cor: eigvec}, which states that the eigencoordinates of $\hat{P}$ or $\tilde{P}$ will approximate those of $P$ (up to a unitary transform) when Theorem \thtwo holds. This suggests that $\hat{z}$, which is computed by spectral clustering of $\hat{P}$, will converge to $z$ up to label permutation. 

%\begin{corollary} \label{cor: eigvec}
%Let Assumptions \asone and \astwo hold, and let $V=(v_1,\ldots,v_K)$ denote the eigenvectors of $P$. Let $\hat{V} = (\hat{v}_1,\ldots,\hat{v}_K)$ denote the eigenvectors of the $K$ largest eigenvalues (in absolute value) of $\hat{P}$ or $\tilde{P}$. Let $\mathcal{O}$ denote the set of $K \times K$ orthogonal matrices. It holds that
%\[  \min_{O \in \mathcal{O}} \|\hat{V}O - V\|_F^2 \leq O_P\left(\frac{1}{\sqrt{n \alpha}}\right).\]
%\end{corollary}

%\section{Proof of Corllary \ref{cor: eigvec}}

\subsection{Intermediate results}

Lemma \ref{le: eigenval} bounds the eigenvalues of $P$ under Assumptions \asone and \astwo.

\begin{customlem}{5}\label{le: eigenval}
	Let Assumptions \asone and \astwo hold. Let $D \in [0,1]^{K \times K}$ denote the matrix $D = \operatorname{diag}(\pi)$. Let $\lambda^*_1, \ldots, \lambda_K^*$ and $\lambda_1, \ldots, \lambda_K$ respectively denote the sorted eigenvalues of $D^{1/2}B^*D^{1/2}$ and $P$. It holds that
	\[ \lambda_k = n\alpha(\lambda^*_k(1 + o_P(1))) \qquad k=1,\ldots,K.\]
\end{customlem}

% Lemma \ref{le: Phat norm} bounds $\|\hat{P}\|_{\operatorname{op}}$ and $\|\tilde{P}\|_{\operatorname{op}}$:
%
% \begin{lemma}\label{le: Phat norm}
% 	It holds that $\|\hat{P}\|_{\operatorname{op}} \leq \alpha n \max_{a,b} B^*_{ab}$ and $\|\tilde{P}\|_{\operatorname{op}} \leq \alpha n \max_{a,b} \hat{B}_{ab}$
% \end{lemma}
	
We will use the following version of the Davis-Kahan theorem, taken from \cite[Th. 4]{yu2015useful}:

\begin{customthm}{3}\label{th: davis kahan}
Let $P, \hat{P} \in \mathbb{R}^{n \times n}$ be symmetric, with singular values $\sigma_1\geq \ldots \geq \sigma_n$ and $\hat{\sigma}_1 \geq \ldots \geq \hat{\sigma}_n$ respectively. Fix $1\leq r \leq s \leq n$ and assume that $\min(\sigma^2_{r-1} - \sigma^2_r, \sigma^2_s - \sigma^2_{s+1}) > 0$, where $\sigma^2_0 = \infty$ and $\sigma_{n+1} = -\infty$. Let $d = s-r+1$, and let $V = (v_r,v_{r+1},\ldots,v_s) \in \mathbb{R}^{n \times d}$ and $\hat{V} = (\hat{v}_r,\hat{v}_{r+1}, \ldots,\hat{v}_s) \in \mathbb{R}^{n \times d}$ have orthonormal columns satisfying $P v_j = \sigma_j u_j$ and $\hat{P}\hat{v}_j = \hat{\sigma}_j\hat{u}_j$ for $j=r,r+1,\ldots,s$. Then there exists orthogonal $\hat{O} \in \mathbb{R}^{d \times d}$ such that
\[ \|\hat{V}\hat{O} - V\|_F \leq \frac{2^{3/2} (2\sigma_1 + \|\hat{P} -P\|_{\operatorname{op}}) \|\hat{P} - P\|_F}{\min(\sigma_{r-1}^2 - \sigma_r^2, \sigma_s^2 - \sigma_{s+1}^2)}\]
\end{customthm}

\subsection{Proof of Corollary \ref{cor: eigvec} and Lemma \ref{le: eigenval}}
\begin{proof}[Proof of Corollary \ref{cor: eigvec}]
	Let $r=1$ and $s = K = \operatorname{rank}(P)$, so that $\sigma_s^2 - \sigma_{s+1}^2 = \lambda_K^2$. 
	By Theorem \ref{th: davis kahan}, it holds that	
\begin{align*}
	\|\hat{V}\hat{O} - V\|_F^2 \leq \left(\frac{2^{3/2} (2\sigma_1 + \|\hat{P} -P\|_{\operatorname{op}}) \|\hat{P} - P\|_F}{\lambda_K^2}\right)^2.
\end{align*}
It follows that 
\begin{align}
\nonumber	\|\hat{V}\hat{O} - V\|_F^2 & \leq \left(\frac{2^{3/2} O(n\alpha) \|\hat{P} - P\|_F}{(n\alpha \lambda_K^*)^2(1+o_P(1))}\right)^2 \\
	& = \left(\frac{2^{3/2} O(1) \|\hat{P} - P\|_F}{n\alpha (\lambda_K^*)^2 (1+o_P(1))}\right)^2, \label{eq: davis kahan 1}
\end{align}
where in the first inequality follows from $\lambda_K = n \alpha \lambda_K^*(1+o_P(1))$ by Lemma \ref{le: eigenval}, and also from 
\begin{align*}
2\sigma_1 + \|\hat{P} - P\|_{\operatorname{op}} & \leq 3 \|P\|_{\operatorname{op}} + \|\hat{P}\|_{\operatorname{op}} \\
& \leq 3n \max_{ij} P_{ij} + n \max_{ij} \hat{P}_{ij} \\
& \leq 3n\alpha \max_{ab} B^*_{ab} + n\alpha \max_{ab} \hat{B}_{ab}. 
\end{align*}
By Theorem \thtwo it holds that
\begin{align}
\nonumber	\frac{1}{n^2} \|\hat{P} - P \|_F^2 & = \frac{1}{n^2} \sum_{i=1}^n \sum_{j=1}^n (\hat{P}_{ij} - P_{ij})^2 \\
\nonumber	& \leq \frac{1}{n^2} \sum_{i=1}^n \sum_{j=1}^n 1\{\hat{P}_{ij} - P_{ij}\} \cdot \max_{i,j} (\hat{P}_{ij} - P_{ij})^2 \\
	& \leq O_P\left(\frac{1}{\sqrt{n\alpha}}\right) \cdot O(\alpha^2). \label{eq: davis kahan 2}
\end{align}
Substituting \eqref{eq: davis kahan 2} into \eqref{eq: davis kahan 1} yields
\begin{align*}
\|\hat{V}\hat{O} - V\|_F^2 \leq \frac{O(1)}{\sqrt{n\alpha}},
\end{align*}
completing the proof.
 \end{proof}
 
\begin{proof}[Proof of Lemma \ref{le: eigenval}]
Let $\hat{\pi} \in [0,1]^K$ be given by
\[ \hat{\pi}_a = \frac{1}{n} \sum_{i=1}^n 1\{z_i = a\},\]
and let $\hat{D} = \operatorname{diag}(\hat{\pi})$. Let $\tilde{u} \in \mathbb{R}^K$ denote an eigenvector of $\hat{D}^{1/2}B^*\hat{D}^{1/2}$ with eigenvalue $\hat{\lambda}$, let $\tilde{v} = \hat{D}^{-1/2}\tilde{u}$, and let $v \in \mathbb{R}^n$ be given by $v_i = \tilde{v}_{z_i}$ for $i \in [n]$.

It can be seen that 
\begin{align*}
	[Pv]_i & = [n \alpha B^* \hat{D} \tilde{v}]_{z_i} \\
	& = [n \alpha B^* \hat{D}^{1/2} \tilde{u}]_{z_i} \\
	& = [n \alpha \hat{D}^{-1/2} \hat{D}^{1/2} B^* \hat{D}^{1/2} \tilde{u}]_{z_i} \\
	& = [n \alpha \hat{D}^{-1/2} \hat{\lambda} \tilde{u}]_{z_i} \\
	& = [n \alpha \hat{\lambda} \tilde{v}]_{z_i} \\
	& = n\alpha \hat{\lambda} v_i,
\end{align*}
showing that $v$ is an eigenvector of $P$ with eigenvalue $n \alpha \hat{\lambda}$. Since $ \hat{D} \rightarrow D$, it follows that the eigenvalues of $\hat{D}^{1/2} B^* \hat{D}^{1/2}$ converge to those of $D^{1/2} B^* D^{1/2}$, completing the proof.
\end{proof}

\bibliographystyle{plain}
\bibliography{bibfile}

\end{document}